\documentclass[11pt, letterpaper, oneside]{article}

\usepackage{latexsym, amsmath, amssymb, amsfonts, amscd}
\usepackage{amsthm}
\usepackage{t1enc}
\usepackage[mathscr]{eucal}
\usepackage{indentfirst}
\usepackage{graphicx, pb-diagram}
\usepackage{fancyhdr}
\usepackage{fancybox}
\usepackage{enumerate}
\usepackage{color}
\usepackage{hyperref}

\usepackage[all,poly,web,knot]{xy}

\theoremstyle{plain}
\newtheorem{thm}{Theorem}[section]

\newtheorem{prop}[thm]{Proposition}
\newtheorem{lemma}[thm]{Lemma}

\renewcommand{\latticebody}{\drop@{ }}

\theoremstyle{definition}
\newtheorem{defi}[thm]{Definition}

\theoremstyle{remark}
\newtheorem{remark}[thm]{Remark}
\newtheorem{ep}[thm]{Example}

\newcommand{\rra}{\rightrightarrows}

\newcommand{\Z}{\ensuremath{\mathbb Z}}

\newcommand{\R}{\ensuremath{\mathbb R}}

\newcommand{\g}{\ensuremath{\frak{g}}}

\newcommand{\h}{\ensuremath{\frak{h}}}

\newcommand{\cL}{{\mathcal L}}

\newcommand{\cA}{\mathcal{A}}
\newcommand{\cE}{\mathcal{E}}
\newcommand{\cC}{\mathcal{C}}

\newcommand{\cG}{\mathcal{G}}

\newcommand{\cO}{\mathcal{O}}
\newcommand{\cM}{\mathcal{M}}

\DeclareMathOperator{\Lie}{Lie}

\newcommand{\bt}{\mathbf{t}}                  
\newcommand{\bs}{\mathbf{s}}                  

\newcommand{\un}{\underline}
\newcommand{\vertprod}{\cdot_v} 
\newcommand{\horprod}{\cdot_h}  

\def\PB(#1,#2,#3,#4){\left\{\begin{matrix}#1&\!\!\!\stackrel{?}{\longrightarrow}&\!\!\!#2\\
\downarrow&&\!\!\!\downarrow\\
#3&\!\!\!\stackrel{?}{\longrightarrow}&\!\!\!#4\end{matrix}\right\}}

\def\pb(#1,#2,#3,#4){ \hom(#1 \to #3, #2 \to #4)}

\newcommand{\pd}[1]{\frac{\partial}{\partial #1}} 

\definecolor{mygray}{rgb}{0.90,0.90,0.95}

\newcommand{\emptycomment}[1]{}
\begin{document}

\title{Higher Lie algebra actions on Lie algebroids}

\author{Marco Zambon\footnote{
Universidad Aut\'onoma de Madrid (Dept. de Matem\'aticas), and ICMAT(CSIC-UAM-UC3M-UCM),
Campus de Cantoblanco,
28049 - Madrid, Spain.
\texttt{marco.zambon@uam.es},\;\texttt{marco.zambon@icmat.es}},
Chenchang Zhu\footnote{Courant Research Centre ``Higher Order Structures'',
University of G\"ottingen, Germany.
\texttt{zhu@uni-math.gwdg.de}}
}

\date{}

\maketitle

\begin{abstract}
We consider a simple instance of action up to homotopy. More precisely, we
 consider strict actions of   DGLAs in degrees $-1$ and $0$ on  degree 1  $NQ$-manifolds. In a more conventional language this means: strict actions of Lie algebra crossed modules on Lie algebroids.

When the action is strict, we show that it integrates to group actions in the categories of Lie algebroids and Lie groupoids (i.e. actions of $\mathcal{LA}$-groups and 2-groups). We perform the integration in the framework of Mackenzie's doubles.

\end{abstract}

\setcounter{tocdepth}{2}

\tableofcontents
\newpage
\section*{Introduction}
\addcontentsline{toc}{section}{Introduction}

In recent years there has been an intense activity integrating certain infinitesimal structures, such as Lie algebroids \cite{cafe}\cite{cf} and $L_{\infty}$-algebras \cite{getzler}\cite{henriques}. Here   ``integration'' is meant in the same sense in which a Lie algebra is integrated to a corresponding  Lie group.
 Both Lie algebroids and (non-positively graded) $L_{\infty}$-algebras are instances of \emph{NQ-manifolds}---non-negatively graded manifolds equipped with a homological vector field (Def. \ref{NQm}).
In this paper we focus on \emph{actions}: we study  infinitesimal actions on NQ-manifolds and their integrations to global actions.
 
We make the following elementary but important observation:
While the   infinitesimal symmetries of an ordinary manifold are controlled by a Lie algebra (the vector fields), the 
infinitesimal symmetries of a NQ-manifold $\cM$ are given by the \emph{differential graded Lie algebra (DGLA)} of 
 vector fields $\chi(\cM)$. Further, if $\cM$ is a NQ-$n$ manifold\footnote{This means that its coordinates  are concentrated in  degrees $0,\dots,n$.}, its infinitesimal symmetries   are controlled by sub-DGLA which is an $(n+1)$-term DGLA---a special kind of Lie $(n+1)$-algebra. This implies that some kind of Lie $(n+1)$-group naturally plays the role of the global symmetries of $\cM$.

An instance of this is given by  the theory of Courant algebroids, which were a source of motivation for this work. Indeed  Courant algebroids are equivalent to a special class of NQ-manifolds (symplectic NQ-2 manifolds \cite{royt}), and  actions on Courant algebroids are realized by more data than just a Lie group action, as showed in  \cite{bcg}. 
\\

In this paper we consider the simplest case $n=1$, that is,   NQ-1 manifolds, which in  ordinary differential geometry language  are just \emph{Lie algebroids}. 
We  saw above that the infinitesimal symmetries of NQ-1 manifolds  are given by  2-term DGLAs. The latter are also known as
\emph{strict Lie 2-algebras}\footnote{For a general study of Lie 2-algebras please see
 Baez et al. \cite{baez:2algebras}.} (Def. \ref{strict}), and they are equivalent to crossed modules of
Lie algebras. 
The integration of a strict Lie
 2-algebra is a strict Lie 2-group. Hence  strict Lie 2-groups control the global symmetries of an NQ-1 manifold.

A \emph{strict (infinitesimal) action} (Def. \ref{strictaction}) on an NQ-1 manifold $\cM$ is 
a morphism of strict Lie 2-algebra
$$L \to \chi(\cM).$$
In \S \ref{sec:act}   we define the notion of \emph{integration} of such an infinitesimal action (Def. \ref{defin}) and we find the integrating action explicitly (Thm. \ref{thm:2groupaction}, Thm. \ref{check}, and Prop. \ref{explicitPhi}): it is an action 
$${\mathcal G}\times \Gamma \rightarrow \Gamma$$
which we describe explicitly, where  ${\mathcal G}$ is the strict Lie-2 group integrating $L$ and  $\Gamma$ is the   Lie groupoid integrating the Lie algebroid corresponding to $\cM$.

 More in detail, 
a strict action of $L$ on $\cM$ gives rise to a
``action double Lie algebroid'', which we integrate  to an ``action double Lie groupoid''. This is a major step, as there are no general statements in the literature which allow to integrate double Lie algebroids to double groupoids.
From the ``action  double Lie groupoid''  we extract the action  ${\mathcal G} \times \Gamma \to \Gamma$.  This gives a conceptual explanation for the results obtained in a special case by Cattaneo and the first author \cite[Thm. 14.1]{CZ2}. Notice that the integrated action is not on $\cM$, but rather on the its integration $\Gamma$.\\

\noindent\textbf{Organization of the text:}
\S \ref{infactNQ} defines NQ-manifolds, strict Lie 2-algebras, and the notion of strict   action. The purpose of  \S \ref{sec:glo} 
is to integrate such strict actions. We introduce the formalism of Mackenzie's doubles, apply it to the integration of strict Lie 2-algebras (a toy example), and finally we integrate  strict actions in \S \ref{sec:act}. In \S \ref{sec:ex} we give  examples of Lie 2-algebra actions on tangent bundles, cotangent bundles of Poisson manifolds and Lie algebras.\\

\noindent\textbf{Notation and conventions:}
$M$ always denotes a smooth manifold.  For any vector bundle $E$, we denote by $E[1]$ the N-manifold obtained from $E$  by declaring that the fiber-wise linear coordinates on $E$ have degree one. If $\cM$ is an N-manifold, we denote by $C(\cM)$
the graded commutative algebras of ``functions on $\cM$''. By
$\chi(\cM)$ we denote graded Lie algebra of vector fields on $\cM$ (i.e., graded derivations of $C(\cM)$).
 
\noindent
The symbol $A$ always denotes a Lie algebroid over $M$. When $A$ is integrable, we denote the corresponding source simply connected Lie groupoid by $\Gamma$, and its source and target maps by $\bs$ and $\bt$. We adopt the convention that two elements $x,y\in \Gamma$ are composable to $x \circ y$ iff $\bs(x)=\bt(y)$. We identify
$A\cong (ker \;\bs_*)|_M$.\\

\noindent\textbf{Acknowledgements:} 
We thank  Rajan Mehta, Tim Porter, Pavol \v{S}evera and Jim Stasheff for their very helpful comments and for discussions. We learnt about the Artin-Mazur
construction from Xiang Tang, whom we hereby thank. We are grateful to him for sharing   
 with us this idea from \cite{MT}. 

\noindent Zambon thanks the Courant Research Centre  ``Higher Order
Structures´´  for hospitality.  Zambon was partially supported by the Centro de Matem\'atica da Universidade do Porto, financed by FCT through the programs POCTI and POSI,  by the FCT program Ciencia 2007, grants PTDC/MAT/098770/2008 and 
PTDC/MAT/099880/2008  (Portugal), and by MICINN RYC-2009-04065 (Spain).

\noindent
Zhu thanks C.R.M. Barcelona for hospitality. Zhu is supported by the German Research Foundation
(Deutsche Forschungsgemeinschaft (DFG)) through
the Institutional Strategy of the University of G\"ottingen.

\section{Infinitesimal actions}\label{infactNQ}
This section introduces NQ-manifolds (\S \ref{grm}) and infinitesimal actions of Lie-2 algebras on them (\S \ref{sec:actions}). The integration of these infinitesimal actions is the main object of this paper. 

\subsection{Background on graded geometry}\label{grm}
 
We start by recalling some background material on graded geometry.
In \S \ref{nnman} we consider degree 1  N-manifolds, which correspond
simply to vector bundles. Then \S \ref{enq} we endow them with a
homological vector field, obtaining Lie algebroids. 

 The reason why we describe classical geometric
objects using graded geometry is that the latter framework allows to
 describe the classical structures and their symmetries in a natural way, namely   in terms of vector fields. Further, since  graded geometry is defined in terms of sheaves, this framework 
 allows to make    use of (graded) local coordinates.
 
\subsubsection{N-manifolds}\label{nnman}
The notion of N-manifold (``N'' stands for non-negative) was introduced by \v{S}evera  in \cite{SWLett}\cite{s:funny}. Here we adopt the definition given by Mehta in  \cite[\S 2]{rajthesis}. Useful references are also  \cite[\S 2]{AlbICM}\cite{AlbFlRio}.

If $V=\oplus_{i< 0}V_i$ is a finite dimensional $\Z_{<0}$-graded vector space, recall that $V^*$ is the $\Z_{>0}$-graded vector space defined by $(V^*)_i=(V_{-i})^*$. We use $S^{\bullet}(V^*)$ to  denote the \emph{graded} symmetric algebra over $V^*$, so its homogeneous elements anti-commute if they both have odd degree.  $S^{\bullet}(V^*)$ is a graded commutative algebra  concentrated in positive degrees.
 
 Ordinary manifolds are modeled on open subsets of $\R^n$, and N-manifolds modeled on the following graded charts:
\begin{defi}\label{locm}
Let $V=\oplus_{i< 0}V_i$ be a finite dimensional $\Z_{<0}$-graded vector space.\\
 The \emph{local model for an N-manifold} consists  of a pair  as follows:
\begin{itemize}
\item $U\subset \R^n$ an open subset 
\item the sheaf (over $U$) of graded commutative algebras given by
$U'\mapsto C^{\infty}(U')\otimes S^{\bullet}(V^*)$.
\end{itemize} 
\end{defi}
\begin{defi}\label{def:grm}
An \emph{N-manifold} $\cM$ consists of a pair  as follows:
\begin{itemize}
\item a topological space $M$ (the ``body'')
\item a sheaf $\cO_M$ over $M$ of graded commutative algebras, locally isomorphic to the
above local model (the sheaf of ``functions'').
\end{itemize}
\end{defi}

We use the notation $C(\cM):=\cO_M(M)$ to denote the space of ``functions
on $\cM$''.
By $C_k(\cM)$ we denote the degree $k$ component of $C(\cM)$, for any non-negative $k$. The
{\em degree} of the graded manifold is the largest $i$ such that  $V_{-i}\neq\{0\}$.
Degree zero graded manifolds are just
 ordinary  manifolds:  $V=\{0\}$, and all functions have degree zero.

\begin{defi}
A \emph{vector field} on $\cM$ is a
graded derivation of the algebra\footnote{Strictly speaking one should define vector fields
in terms of the sheaf $\cO_M$ over $M$. However we will work only with objects defined on the whole of the body $M$, hence the above definition will suffice for our purposes. }
$C(\cM)$.
\end{defi}

Since $C(\cM)$ is a graded commutative algebra (concentrated in non-negative degrees), the space of vector fields $\chi(\cM)$,
equipped with the graded commutator $[-,-]$,  is a graded
Lie algebra (see Def. \ref{dgla}). 

All N-manifolds in this note arise from graded vector bundles\footnote{Actually it can be shown that every finite dimensional N-manifold is (non-canonically) isomorphic to one arising from a graded vector bundle.}, as follows:
\begin{ep}\label{VBs}
Let $F=\oplus_{i< 0}F_{i} \rightarrow M$ be a   graded vector bundle.
The N-manifold associated to it has body $M$, and $\cO_M$ is given by the sheaf of sections of  $S^{\bullet}F^*$. \end{ep}

We will focus mainly on  degree 1 N-manifolds, which we now describe in more detail. To do so we recall first
\begin{defi}\label{CDO}
Given a vector bundle $E$ over $M$, a  \emph{covariant differential operator\footnote{Also known as derivative endomorphism, see \cite[\S 1]{MC-YKS}.} (CDO)}
is a linear map $Y: \Gamma(E) \to \Gamma(E) $ such that there exists a vector field
$\underline{Y} $  on $M$ (called \emph{symbol}) with
\begin{align}\label{eqCDO}
 Y(f\cdot e) = \underline{Y} (f) e + f \cdot Y(e), \quad
\text{for} \; f \in C^\infty(M), \; e \in \Gamma(E).
\end{align}
We denote the set of CDOs on $E$ by
 $CDO(E)$.
If $Y\in CDO(E)$, then
the dual  $Y^*\in CDO(E^*)$ is defined by
\begin{equation}\label{dualCDO}
\langle Y^* (\xi), e \rangle + \langle \xi, Y(e) \rangle=\un{Y}
(\langle \xi, e\rangle ), \quad \text{for all}\; e \in \Gamma(E), \;
\xi \in \Gamma(E^*).
\end{equation}
\end{defi}

Recall that if $E\rightarrow M$ is a (ordinary)  vector bundle, $E[1]$ denotes the graded vector bundle whose fiber over $x\in M$ is  $(E_x)[1]$ (a graded vector space concentrated in degree $-1$).

\begin{lemma}\label{N1}
If $E\rightarrow M$ is a  vector bundle, then
$\cM:=E[1] $ is a degree $1$
N-manifold with body $M$, and conversely all  degree $1$
N-manifolds arise this way.

The algebra of functions $C(\cM)$ is generated
by
$$C_0(\cM)=C^{\infty}(M) \text{  and  }
C_1(\cM)= \Gamma(E^*).$$

The $C(\cM)$-module of vector fields is generated by elements in degrees $-1$ and $0$. We have identifications
$$\chi_{-1}(\cM)=\Gamma(E)\text{  and  }\chi_0(\cM)=CDO(E^*)$$
 induced by the actions on functions. Further the map $\chi_0(\cM) \cong CDO(E)$  obtained dualizing CDOs is just $X_0\mapsto[X_0,\cdot]$ (using the identification   $\chi_{-1}(\cM)=\Gamma(E)$).
\end{lemma}
The proof of Lemma \ref{N1} is given in Appendix \ref{app:nq}.

\begin{remark}\label{coordsNQ1}
Let us choose
 coordinates $\{x_i\}$ on an open subset $U\subset M$ and a  frame $\{e_{\alpha}\}$ of sections of $E|_U$.
Let $\{\xi^{\alpha}\}$
be the dual frame for $E^*|_U$, and assign degree $1$ to its elements. Then $\{x_i, \xi^{\alpha}\}$
form a set of coordinates for $\cM:=E[1]$ (in particular they generate $C(\cM)$ over $U$).
The coordinate expression of vector fields is as follows. $\chi_{-1}(\cM)$ consist of elements of the form $f_{\alpha}\pd{\xi^{\alpha}}$, and $\chi_{0}(\cM)$ of elements of the form $g_{i}\pd{x_i}+
f_{ \alpha \beta}\xi^{\alpha}\pd{\xi^{\beta}}.$ Here  $f_{\alpha},g_{i},f_{ \alpha \beta} \in C^{\infty}(M)$, for $i\le dim(M)$ and $\alpha, \beta \le rk(E)$, and we adopt the Einstein summation convention.
\end{remark}

\subsubsection{NQ-manifolds and Lie algebroids}\label{enq}
We will be interested in N-manifolds equipped with extra structure:
\begin{defi}\label{NQm}
An {\em NQ-manifold} is an N-manifold $\cM$  equipped
with a \emph{homological vector field}, i.e. a degree 1 vector field $Q$ such that $[Q, Q]=0$.
\end{defi}
 
To shorten notation, we call a degree $n$ NQ-manifold a {\em NQ-$n$ manifold}.

Before considering $\chi(\cM)$, we recall the notion of  differential graded Lie algebra (DGLA):
\begin{defi}\label{dgla}
  A \emph{graded Lie algebra} consists of a  graded vector space $L=\oplus_{i\in \Z} L_i$
together with a  bilinear bracket $[\cdot,\cdot] \colon L \times L \rightarrow L$ such that, for all  homogeneous  $a,b,c\in L$:
\begin{itemize}
\item[--] the bracket is degree-preserving: $[L_i,L_j]\subset L_{i+j}$
 \item[--] the bracket is graded skew-symmetric:
 $[a,b]=-(-1)^{|a||b|}[b,a]$
\item[--] the adjoint action $[a,\cdot]$
is a degree $|a|$ derivation of the bracket (Jacobi identity): $[a,[b,c]]=[[a,b],c]+(-1)^{|a||b|}[b,[a,c]]$ .
\end{itemize}
  A \emph{differential graded Lie algebra} (DGLA) $(L,[\cdot,\cdot], \delta)$ is a graded Lie algebra together with a linear $\delta : L \rightarrow L$ such that
\begin{itemize}
 \item[--] $\delta$ is a degree $1$ derivation of the bracket:
$\delta(L_i)\subset L_{i+1}$ and $\delta[a,b]=[\delta a,b]+ (-1)^{|a|}[a, \delta b]$
\item[--]  $\delta^2=0$.
\end{itemize}
\end{defi}

\begin{lemma} \label{lemma:vf}
For a NQ-$n$ manifold  $\cM$, the space of vector fields  $$(\chi(\cM),  [Q, -], [-,-])$$
is a negatively bounded DGLA with lowest degree $-n$.
\end{lemma}
\begin{proof} The fact that $[Q, -]$   squares to zero follows from $[Q,[Q,-]]=\frac{1}{2}[[Q, Q],-]=0$. The fact that $[Q, -]$ is a degree $1$ derivation of the Lie bracket follows from the Jacobi identity. So the above is a DGLA.

 A vector field on $\cM$ has local expression $\sum_{i} f_i \frac{\partial}{\partial y_i}$, where $f_i\in C(\cM)$ and $y_i$'s are local coordinates on $\cM$. The degree of $\frac{\partial}{\partial y_i}$ is $-deg(y_i)$. Since $deg(y_i)\in\{0,\dots,n\}$ we are done.\end{proof}

\begin{remark}\label{rk:n-term}
As a $C(\cM)$-module, $\chi(\cM)$ is generated by its elements in degrees $-n,\dots,0$.
This suggests  that the most important information is contained in following the truncated DGLA (a sub-DGLA of $\chi(\cM)$):
\begin{equation} \label{eq:vf-m}
\chi_{-n} (\cM)\oplus  \cdots \oplus \chi_{-1} (\cM)\oplus \chi_{0}^Q (\cM),
\end{equation}
where
\begin{equation*} 
 \chi_{0}^Q (\cM):= \{X \in \chi_{0}(\cM): [Q,X]=0 \}.
\end{equation*}
\end{remark}

A well-known example of NQ-manifolds is given by Lie algebroids \cite{MK2}.

\begin{defi}
A {\em Lie algebroid} $A$ over a manifold $M$ is a vector
bundle over $M$, such that the global sections of $A$ form a Lie
algebra with  Lie bracket $[\cdot,\cdot]_A$ and Leibniz rule holds
\[ [X, fY]_A=f[X, Y]_A + \rho(X)(f) Y , \quad X, Y \in \Gamma(A), f
\in C^\infty(M), \]where $\rho: A \to TM$  is
  a vector
bundle morphism called the  \emph{anchor}.
\end{defi}

The following is well known (\cite{vaintrob}, see also \cite{yvette}):
\begin{lemma}\label{nq1la}
NQ-1 manifolds  are in bijective correspondence with Lie algebroids.
\end{lemma}

We describe the  correspondence using the derived bracket
construction.
By Lemma \ref{N1} there is a bijection between vector bundles and
degree 1 N-manifolds. If $A$ is a Lie algebroid, then the homological vector field  is just the   Lie algebroid differential acting on $ \Gamma(\wedge^{\bullet}A^*)=C(A[1])$. Conversely, if
$(\cM:=A[1], Q_A)$ is an NQ-manifold, then
the Lie algebroid structure on $A$ can be recovered by the derived bracket
construction \cite[\S 4.3]{yvette}: using the
 identification $\chi_{-1}(\cM)= \Gamma(A)$ recalled in Lemma \ref{N1}, we define
\begin{equation}\label{eq:br-rho}
[a,a']_A=[[Q_A,a],a'], \;\;\;\;\quad  \rho(a)f=[[Q_A,a],f],
\end{equation}
where $a,a'\in \Gamma(A)$ and $f\in C^{\infty}(M)$.

In coordinates the correspondence is as follows.
Choose coordinates $x_{\alpha}$ on $M$ and a frame of sections $e_i$ of $A$, inducing (degree $1$) coordinates $\xi_i$ on the fibers of $A[1]$. Then
\begin{equation}\label{QA}
Q_A=\frac{1}{2} \xi^j\xi^i
c_{ij}^k(x)\pd{\xi_k}+\rho_{i}^{\alpha}(x)\xi^i\pd{x_{\alpha}}
\end{equation}
where $[e_i,e_j]_A=c_{ij}^k(x)e_k$ and the anchor of $e_i$ is $\rho_{i}^{\alpha}(x)\pd{x_{\alpha}}$.

 Viewing Lie algebroids as NQ-manifolds proves to be very valuable. For
example, the definition of  {\em Lie
algebroid morphism} $A\to A'$ is quite involved, but in terms of NQ-manifolds it is simply a morphism of N-manifolds from $A[1]$ to $A'[1]$ (i.e., a
   morphism of graded commutative algebras $C(A'[1]) \to C(A[1])$) which respects
homological vector field. Similarly, the notion of double Lie
algebroid is quite involved, but it simplifies once expressed in terms
of homological vector fields (Def. \ref{doublealgoid}).

\subsection{Strict Lie 2-algebra actions on NQ-1 manifolds}\label{sec:actions}

In this subsection we define strict Lie 2-algebra actions on   NQ-1
manifolds in \S \ref{defact}.  Then in \S \ref{subsec:vfnq1} we study actions from an ordinary differential geometry view point, that is, we view NQ-1 manifolds as Lie algebroids  and describe the action  by ordinary differential geometry data.

\subsubsection{Definition of strict Lie 2-algebra action}\label{defact}
 
\begin{defi}\label{strict}
A \emph{strict Lie 2-algebra} (in the sense of \cite{baez:2algebras}) is a DGLA (see Def. \ref{dgla}) concentrated in degrees $-1$ and $0$.
\end{defi}
Concretely,  a strict Lie 2-algebra can be described as follows. It is given by $\h[1]\oplus \g$ (where $\h,\g$ are ordinary vector space) together with a Lie algebra structure on $\g$, a left $\g$-module structure on $\h$ (both of which we denote by $[\cdot,\cdot]$, and a linear map $\delta \colon \h \to \g$ satisfying $\delta[v,w]=[v, \delta w]$ for all $v\in \g, w\in \h$. 

\begin{defi} \label{strictaction}
A \emph{strict} action of a strict Lie 2-algebra $\h[1]\oplus \g$ on an NQ-1 manifold $\cM$ is a morphism of DGLAs
$$\mu: \h[1]\oplus \g \to \chi(\cM),$$
i.e., a degree-preserving linear map
preserving the differentials and Lie brackets.
\end{defi}

The NQ-1 manifold $\cM$ is equal to $A[1]$ for some Lie algebroid $A$, by Lemma \ref{nq1la}.
We spell out what Def. \ref{strictaction} means: we have maps
\begin{align*}
\mu|_{\h[1]} \colon \h[1] \to& \chi_{-1}(A[1])\\
\mu|_{\g} \colon \g \to& \chi_{0}(A[1])
\end{align*}
such that
\begin{align}
\label{c}\mu(\delta w)=&d_Q(\mu (w)) \;\;\;\;\;\;\; \text{ for all }w\in \h[1],\\
\label{d}0=&d_Q(\mu (v))  \;\;\;\;\;\;\;\;\text{ for all }v\in \g,\\
\label{a}\mu[v,w]=&[\mu (v),\mu (w)]  \;\;\;\text{ for all }v\in \g,w\in \h[1],\\
\label{b}\mu[v_1,v_2]=&[\mu (v_1),\mu (v_2)]\;   \text{ for all }v_i\in \g.
\end{align}

\begin{remark}
  By eq. \eqref{c}, the image of the action map $\mu$ will be contained in the truncated DGLA
$\chi_{-1}(\cM) \oplus \chi_0^Q(\cM)$ (see eq.
 \eqref{eq:vf-m}).  Hence  Lie 2-algebra
actions on $\cM$ can be formulated using only the truncated DGLA.
\end{remark}

\subsubsection{DGLAs and  Lie algebra crossed modules}\label{subsec:vfnq1}

Let $\h[1]\oplus \g$ be a strict Lie 2-algebra. Recall that $[w,w']_{\delta}:=[\delta w,w']$ makes $\h$ into a Lie algebra, which we denote by $\h_{\delta}$.
Let $A[1]$ be an NQ-1 manifold. Using the identifications $\chi_{-1}(A[1])= \Gamma(A)$ and
$ \chi_{0}(A[1])\cong CDO(A)$ given in Lemma \ref{N1}, we obtain a characterization of strict actions (Def. \ref{strictaction}) in terms of classical geometric objects.

\begin{lemma}\label{classical} Let  $\mu: \h[1]\oplus \g \to \chi(\cM)$ be   a linear map.
Then $\mu$ is a morphism of DGLAs
iff
\begin{itemize}
\item  $\mu|_{\g} \colon \g \to  CDO(A)$ is an infinitesimal action of $\g$ on $A$ by infinitesimal Lie algebroid automorphisms, with the property that $\delta(w)$ acts as $[\mu (w),\cdot]_A$ for all $w\in \h$
\item  $\mu|_{\h} \colon \h_{\delta} \to \Gamma(A)$
is a Lie algebra homomorphism which is equivariant w.r.t. the representation of $\g$ on $\h$ by $v\mapsto [v,\cdot]$ and the representation $\mu|_{\g}$ of $\g$ on $\Gamma(A)$.
\end{itemize}
\end{lemma}
\begin{proof} Assume that $\mu$ is a strict action.
\eqref{b} means that $\mu|_{\g}$ is an infinitesimal action of $\g$ on the vector bundle  $A$, and \eqref{d} means that the action is by infinitesimal Lie algebroid automorphisms.

\eqref{c} means that $\delta w$ acts by  $[\mu (w),\cdot]_A$ for all $w\in \h$, by the derived bracket construction.

\eqref{a} means that $\mu|_{\h}$ is an equivariant map.

Further $\mu|_{\h}$ is a Lie algebra morphism:$$\mu[w_1,w_2]_{\delta}=\mu[\delta w_1,w_2]=[\mu(\delta w_1),\mu(w_2)]=
[[Q,\mu( w_1)],\mu(w_2)]=[\mu(w_1),\mu(w_2)]_A,$$
where  the second equality holds by \eqref{a}  and the third equality   by  \eqref{c}.

The converse implication is obtained reversing the argument.
\end{proof}

The idea behind  Lemma \ref{classical} is the concept of  crossed
module, which might  be more familiar to the reader than strict Lie
2-algebras, even though they are equivalent concepts (see Prop
\ref{bij}).
\begin{defi}
A \emph{crossed module of Lie algebras} $(\h,\g,\delta,\alpha)$
consists  of a Lie algebra morphism $\delta: \h \to \g $ and an
(left) action  of $\g$ on $\h$ by derivations, i.e. $\alpha : \g \to
Der(\h) $, such that
\[
\delta (\alpha  (v) (w)) = [ v, \delta (w)], \quad \quad \alpha ( \delta (w) ) w' =
[w, w'].
\]
where $v\in \g$, $w,w'\in \h$.
\end{defi}

A classical result  (see \cite[Thm.
36]{baez:2algebras}) is that
\begin{prop}\label{bij}
There is a one-to-one correspondence between  strict Lie 2-algebras     and
crossed modules of  Lie algebra.
\end{prop}
Since a similar construction will show up again  in \S \ref{sec:lag}, we recall the  correspondence: A
strict Lie 2-algebras $(L_{-1}[1]\oplus L_0,\delta, [\cdot,\cdot])$ gives rise to a Lie algebra crossed
module with $\h = L_{-1}$ and $\g= L_{0}$ where
\[
\begin{split}
[w, w']_{\h} &:= [ \delta (w), w'] , \\
\alpha  (v) (w) & := [v, w],\\
[v, v']_{\g} & := [v, v']
\end{split}
\]
for $v,v'\in \g$, $w,w'\in \h$.
The Jacobi identity of $[\cdot,\cdot]$ gives the Jacobi identities of $[\cdot,\cdot]_\h$
and $[\cdot,\cdot]_\g$ and the remaining conditions for crossed modules.

On the other hand, a crossed module $(\h,\g,\delta,\alpha)$ gives rise to
a strict Lie-2 algebra  with
  $L_{-1}=\h$, $L_0 = \g$, $\delta$ as a differential, and
\[
\begin{split}
[w, w'] & :=0, \\
[v, v'] & := [v, v']_{\g}, \\
[v, w] & = -[w,v] := \alpha  (v) (w)
\end{split}
\]
for $v,v'\in L_0$, $w,w'\in L_{-1}$.
The Jacobi identity of $[\cdot,\cdot]$ is implied by the Jacobi identities of $[\cdot,\cdot]_\h$
and $[\cdot,\cdot]_\g$ and various conditions for crossed modules.  

Consequently, 
there is a 1-1 correspondence between DGLA morphisms and crossed module morphisms. Thus Lemma \ref{classical} basically explicitly tells us that given  $\mu: \h[1]\oplus \g \to \chi(\cM)$ a linear map, then $\mu$ is a morphism of DGLAs
iff $(\mu|_{\h},  \mu|_{\g})$ is a morphism of Lie algebra crossed modules from the crossed module  associated to  $\h[1]\oplus \g$ to the one associated to the the truncated DGLA $\chi_{-1}(A[1])\oplus\chi_{0}(A[1])^Q$.
For this, we only need to notice that the crossed module associated to   $\chi_{-1}(A[1])\oplus\chi_{0}(A[1])^Q$ by Lemma \ref{bij} is the quadruple given by $(\Gamma(A),[\cdot,\cdot]_A)$, the subset of $CDO(A)$ consisting of infinitesimal Lie algebroid automorphisms, the morphism $\delta(a)=[a,\cdot]_A$ and the natural action of $CDO(A)$ on $\Gamma(A)$.

\section{Integration to global actions}\label{sec:glo}

The purpose of this section is to integrate the infinitesimal actions introduced in Def. \ref{strictaction}. We do so using the framework of Mackenzie's doubles, which we review in \S \ref{Mac}. In \S \ref{sec:lag} we display the objects integrating strict Lie 2-algebras, and in \S \ref{sec:act} -- the heart of this paper --  we integrate the corresponding strict actions.

Recall that one can differentiate a Lie groupoid $G_1 \rightrightarrows G_0$ to obtain
its Lie algebroid. This defines a functor from the category of
Lie groupoids to the category of Lie algebroids, called Lie functor. We refer to the inverse
process as ``integration''.

\subsection{Background on Mackenzie's doubles}\label{Mac}

In this subsection we recall the formalism of Mackenzie doubles and its extension by Mehta. We will use it in \S \ref{sec:lag} and \S \ref{sec:act} to integrate strict Lie 2-algebras and their actions.

Recall   that  one can apply the Lie functor to any Lie groupoid  to obtain its Lie algebroid.
Further, given a Lie algebroid, applying the degree shifting functor $[1]$ one obtains an NQ-1 manifold (Lemma \ref{nq1la}).
The formalism of Mackenzie doubles relates \emph{double Lie groupoids} (see Def. \ref{strlgd})
 to three other structures obtained  applying (horizontally or vertically) the Lie functor.
This was extended by Mehta \cite{rajthesis} who
applied (horizontally or vertically) the degree shifting functor $[1]$ to obtain NQ-manifolds with additional structures.
 The situation is summarized in the following diagram taken from  \cite{rajthesis}:
\begin{equation}\label{rajdiaWHOLE}
\xymatrix{
\mbox{Double Lie groupoids} \ar^{\Lie_H}[r] \ar^{\Lie_V}[d] & \mbox{$\mathcal{LA}$-groupoids} \ar^{[1]}[r] \ar^{\Lie}[d] & \mbox{$Q$-groupoids} \ar^{\Lie}[d] \\
\mbox{$\mathcal{LA}$-groupoids} \ar^{\Lie}[r] \ar^{[1]}[d] & \mbox{Double Lie algebroids} \ar^{[1]_H}[r] \ar^{[1]_V}[d] & \mbox{$Q$-algebroids} \ar^{[1]}[d] \\
\mbox{$Q$-groupoids} \ar^{\Lie}[r] & \mbox{$Q$-algebroids} \ar^{[1]}[r] & \mbox{Double $Q$-manifolds}
}
\end{equation}

\begin{remark}\label{invert}
In general it is not known whether the $\Lie$ functors appearing in the above diagram can  be inverted.
For instance, given a double Lie algebroid whose vertical Lie algebroids are integrable to Lie groupoids, it is not known if the integrating Lie groupoids form an $\mathcal{LA}$-groupoid.
The following question is also open: does a  $\mathcal{LA}$-groupoid for which the Lie algebroid structures are
integrable arise from a double groupoid? Partial answers to this problem were worked out in \cite{lucathesis}.
\end{remark}

The portion of diagram \eqref{rajdiaWHOLE} which is relevant to us is the following:
\begin{equation}\label{rajdia}
\colorbox{mygray}{
\xymatrix{
\mbox{Double Lie groupoids} \ar^{\Lie_H}[r]  & \mbox{$\mathcal{LA}$-groupoids}   \ar^{\Lie}[d] &   \\
 & \mbox{Double Lie algebroids} \ar^{\;\;\;\;\;[1]_H}[r] & \mbox{$Q$-algebroids}  }
      }
\end{equation}
We define the  objects appearing  in it.
We point out that Mehta \cite{rajthesis} works entirely in the category of graded manifolds, whereas we want to assume that the double Lie groupoids appearing in
\eqref{rajdia} consist of  ordinary manifolds. This explain why our definitions below are more restrictive than those of  \cite{rajthesis}.

\begin{defi}\label{strlgd} Let {\bf StrLgd} be the category of Lie groupoids with
  strict morphisms\footnote{That is, usual Lie groupoid
    morphisms.  We will not make use of the notion of generalized morphism, i.e.    Hilsum-Skandalis bimodule.}. A
\emph{double Lie groupoid} is a groupoid object in
{\bf StrLgd}. A
 { \em strict Lie 2-group} is a group object in {\bf StrLgd}. A \emph{strict Lie 2-group action} is a group action in {\bf StrLgd}.
\end{defi}

\begin{defi}\label{Lagrp}  Let {\bf LA} be the category of Lie algebroid and Lie algebroid morphisms.  A \emph{$\mathcal{LA}$-groupoid} is a groupoid object in
{\bf LA}.
An \emph{$\mathcal{LA}$-group} is a group\ object in  {\bf LA}. An \emph{$\mathcal{LA}$-group action} on
a Lie algebroid is a group action in {\bf LA}.\end{defi}

 In plain English, an $\mathcal{LA}$-group is a Lie
algebroid $C$ endowed with an additional group structure. For
example, there is a multiplication
\[ m: C \times C \to C \]
which is a Lie algebroid morphism and satisfies the (strict) associativity diagram.  There are also an identity
morphism and an inverse morphism
\[ e: pt \to C, \quad i: C \to C, \]
which satisfy the group axioms. If $C$ is a Lie algebroid over $N$, then these axioms say exactly that
both $C$ and $N$ are Lie groups, and $C \to N$ is a group morphism, i.e. $N$ has an induced group
structure from $C$. 

\begin{remark} A Q-groupoid is a groupoid object in the category of NQ-1 manifolds. We will not make use of this notion.
By the correspondence between Lie algebroids and NQ-1 manifolds (see Lemma \ref{nq1la}) it is clear that $\mathcal{LA}$-group(oid)s correspond to
Q-group(oid)s.\end{remark}

\begin{defi} \label{doublealgoid}
A \emph{double Lie algebroid} \cite{MacKdoublealgebroid}  is a double vector bundle
\begin{equation}
\xymatrix{
D\ar[d]  \ar[r] & B\ar[d] \\
A \ar[r]& M}
\end{equation}
such that both vertical sides and both horizontal sides are Lie
algebroids, subject to certain compatibility conditions (see for
instance \cite[\S 1]{vorover}). By  \cite[Thm. 1]{vorover}, the
compatibility conditions are equivalent to the following condition. If
we apply the $[1]$-functor to the vertical sides, to obtain a vector
bundle of graded manifolds $D[1]_A \to B[1]$, and then we apply again
the $[1]$-functor to it, the resulting degree $2$ graded manifold
$(D[1]_A)[1]$ will be endowed with homological vector fields
encoding the Lie algebroid structures on $D\to A$ and $D\to B$; the
condition is that  these two vector fields commute.
\end{defi}

\begin{defi} \label{Q-algoid}
A \emph{Q-algebroid} \cite[Def. 4.22]{rajQalg}  is a {\bf N1}-algebroid\footnote{Mehta refers to it as ``superalgebroid'' and allows for $\Z$-graded manifolds.}
\cite[ Def. 4.1]{rajQalg} $\cA\rightarrow \cM$ with a homological
vector field $Q$ which is \emph{morphic} \cite[ Def. 4.14]{rajQalg}.
A Q-algebroid for which $\cM$ is a point is called  a
\emph{Q-algebra}.
\end{defi}

Q-algebras are exactly the same thing as strict Lie 2-algebras, as we will show in Lemma \ref{DGLAdoubleQ}.

Definition \ref{Q-algoid}   needs some explanation. An {\bf N1}-vector
bundle\footnote{Mehta \cite{rajQalg} refers to it as ``super vector bundle'' and allows for arbitrary $\Z$-graded manifolds. In our notation, the prefix ``{\bf N1}'' refers to the category of degree 1 N-manifolds.} 
  \cite[ Def. 2.1]{rajQalg} $\cE\to \cM$ consists of two degree 1  N-manifolds and a surjection between them, subject to a trivialization condition.
  A section \cite[ Def. 2.4]{rajQalg} is just a map of N-manifolds $s \colon \cM\to \cE$ (not necessarily degree-preserving) which composed with the projection equals $Id_{\cM}$. A {\bf N1}-vector bundle $\cE\to \cM$, endowed with a  (degree preserving)
 morphism from $\Gamma(\cE) \to \chi(\cM)$ and bracket $\Gamma(\cE)\times \Gamma(\cE) \to \Gamma(\cE)$ satisfying the graded Leibniz rule and Jacobi identity,
forms a {\bf N1}-algebroid \cite[ Def. 4.1]{rajQalg}. An example  \cite[Ex. 2.4.4]{rajthesis} is given by the action algebroid induced from a morphism of graded Lie algebras into $\chi(\cM)$, for $\cM$ a degree 1 N-manifold. Graded Lie algebras concentrated in  degrees $-1$ and $0$ are exactly the  {\bf N1}-algebroids for which the base $\cM$ is a point.

A vector field  on a  {\bf N1}-algebroid  $\cA\rightarrow \cM$ is
\emph{morphic} \cite[ Def. 4.14]{rajQalg} if it is linear in on the
fibers (in the sense that its action on a function linear on the
fibers is linear again)  and, viewed as a vector field on $\cA[1]$,
it commutes with $Q_{\cA}$.  Here $Q_{\cA}$ is the homological
vector field on $\cA[1]$ which, by virtue of \cite[Thm.
4.6]{rajQalg}, encodes the  {\bf N1}-algebroid structure on
$\cA\rightarrow \cM$.


\subsection{Integrating strict Lie 2-algebras: $\mathcal{LA}$-groups and 2-groups}\label{sec:lag}
In this subsection we describe the objects that integrate a strict Lie-2 algebra (Def. \ref{strict}), namely strict 2-groups (Def. \ref{strlgd}) and $\mathcal{LA}$-groups (Def. \ref{Lagrp}).  
The idea is  
\begin{itemize}
\item   \S \ref{whatsl2}: show that strict Lie-2 algebras fit in the framework of diagram \eqref{rajdia}  
\item  \S \ref{intDGLA}: chase back in   diagram \eqref{rajdia} to obtain the integrating objects. 
\end{itemize}

\subsubsection{What is an integration of a strict Lie 2-algebra?}\label{whatsl2}

An integration of a strict Lie 2-algebra $\h[1]\oplus {\g}$ is {\em defined} to be a
strict Lie 2-group \cite{baez:2gp} whose corresponding crossed module differentiates
to the  crossed module corresponding to $\h[1]\oplus {\g}$  (Lemma \ref{bij}). 

Here we provide an integration procedure by 
viewing a strict Lie
2-algebras as a $Q$-algebroid\footnote{Equivalently
we could view strict Lie 2-algebras  as
double Q-manifolds or as double Lie algebroids.} (Lemma
\ref{DGLAdoubleQ}). This point of view has the advantage of  providing a  hint for what the right notion of integration of a strict action of a strict Lie 2-algebra  should be (see Def. \ref{strictaction}).

\emptycomment{
More precisely, in this subsection we show that strict Lie
2-algebras can be viewed as a $Q$-algebroid\footnote{Equivalently
we could view strict Lie 2-algebras  as
double Q-manifolds or as double Lie algebroids.} (Lemma \ref{DGLAdoubleQ}). We then \emph{define} their integration to be a double groupoid or $\mathcal{LA}$-groupoid which, upon applying the relevant functors in the diagram  \eqref{rajdia}, give back the original $Q$-algebroid.
We will identify the integrations in the next subsection by ``chasing back'' in   diagram  \eqref{rajdia}.}

Given a finite dimensional $\Z$-graded vector space $L$,
a DGLA structure
 on $L$ can be encoded by means of a  homological
  vector field $Q$ on $L[1]$ (see Def. \ref{NQm}).
We recall
the procedure to recover the DGLA structure on   $L$ from $Q$; it is a special case of Voronov's \emph{higher derived brackets} construction  \cite[Ex. 4.1]{voro}. We have
\begin{equation*}\delta v= [Q,\iota_{v}]|_0\;\;\;\;\;\text{             and               } \;\;\;\;\;[v_1,v_2]_L=(-1)^{|v_1|}[[Q,\iota_{v_1}],\iota_{v_2}]|_0
\end{equation*}
where $|v|$ is the degree of $v\in L$,  $\iota_{v}$ is the induced constant vector field   on $L[1]$,  and ``$|_0$'' denotes the evaluation at the origin in $L[1]$.

 \begin{remark}\label{linqaud} To make things more explicit, we express the above construction in
 coordinates.
If $\{e_i\}$ is a basis of $L$ and $\xi^i$ are the corresponding coordinates on $L[1]$, then $Q$ is at most quadratic:
\begin{align}
Q= Q_{\delta}+Q_{br}=\left(\xi^iQ^k_{i}+\frac{1}{2}\xi^j\xi^iQ^k_{ij} \right) \pd{\xi^k}
\end{align}
where we denote by
 $Q_{\delta}$ and $Q_{br}$ respectively  the linear and quadratic component of $Q$.
One has
\begin{equation}\label{qd}
\delta e_i=(-1)^{|e_i|}Q^k_{i}e_k \;\;\;\;\;\;\text{             and               }\;\;\;\;\;\;\;\; [e_i,e_j]_L=(-1)^{|e_j|}Q^k_{ij}e_k.
\end{equation}
\end{remark}

\begin{lemma}\label{DGLAdoubleQ} There is a bijection between strict Lie 2-algebras   and Q-algebras,
given by $(L, [\cdot,\cdot]_L, \delta)\mapsto (L, [\cdot,\cdot]_L, -Q_{\delta})$.

\end{lemma}
\begin{proof}
Recall from Def. \ref{Q-algoid} that a Q-algebra is a graded Lie algebra concentrated in degrees $-1$ and $0$, together with a morphic vector field. 

Let  $(L, [\cdot,\cdot]_L, \delta)$ be a strict Lie 2-algebra.
All three of $Q_{br}$, $Q_{\delta}$ and $Q_{br}+Q_{\delta}$ are self-commuting, because  by the derived bracket construction they define $L_{\infty}$-structures on $L$ (namely the graded Lie bracket $ [\cdot,\cdot]_L$, the differential $\delta,$ and the DGLA structure).
In particular $Q_{\delta}$ is a linear homological vector field on $L[1]$. Further 
 $$[Q_{br}+Q_{\delta},Q_{br}+Q_{\delta}]=[Q_{br},Q_{br}]+[Q_{\delta},Q_{\delta}]+2[Q_{br},Q_{\delta}]$$ implies that $[Q_{br},Q_{\delta}]=0$,
 that is, $Q_{\delta}$ is a morphic vector field. Clearly $-Q_{\delta}$ is also a morphic vector field, so $(L, [\cdot,\cdot]_L, -Q_{\delta})$ is a Q-algebra. Reversing the argument we see that all Q-algebras arise this way.
\end{proof}

\subsubsection{Integrating the strict Lie 2-algebra}\label{intDGLA}

Let $\h[1]\oplus \g$ be a strict Lie 2-algebra. By Lemma \ref{DGLAdoubleQ} we can view it as  a $Q$-algebra. In this subsection we argue that  it
lies in the image of the functors appearing in the diagram \eqref{rajdia},
as follows\footnote{The notation is chosen so to describe the \emph{vertical} (Lie group or Lie algebra) structures. The bottom horizontal sides of the  double Lie groupoid, $\mathcal{LA}$-groupoid, etc appearing in diagram   \eqref{rajdiaspecial} are just points, so they are omitted.}:
\begin{equation}\label{rajdiaspecial}
\colorbox{mygray}{
\xymatrix{
(H \rtimes G)\rightrightarrows G  \;\;\;\ar^{\Lie_H}[r]
&
\;\;\; (\h\rtimes G) \rightarrow   G
\ar^{\Lie}[d] &
\\
&
\;\;\;\;\;\;  (\h  \rtimes \g)  \rightarrow \g \;\;\;\ar^{[1]_H}[r]
&
\;\;\;(\h[1] \rtimes \g, -Q_{\delta}).
&
&
}
}
\end{equation}
Recall that the integration of the strict Lie 2-algebra $\h[1]\oplus \g$ is by definition  the strict  Lie 2-group in the upper left corner.



We describe the structures appearing in diagram \eqref{rajdiaspecial}, in particular
the strict  Lie 2-group in the upper left corner.
The strict Lie 2-algebra $\h[1]\oplus \g$ corresponds to the crossed module of Lie algebras $(\g,\h_{\delta},\delta, \alpha)$ (Lemma \ref{bij}), where $\h_{\delta}$ denotes the Lie algebra structure on the vector space $\h$ with bracket $[w_1,w_2]_{\delta}:=[\delta w_1,w_2]$ and $\alpha(v)=[v,\cdot]$ for $v\in \g$. Consider the quadruple\footnote{This quadruple forms what is known as the
\emph{crossed module of Lie groups}  integrating the above crossed module of Lie algebras.}
 $(G,H,\bt,\phi)$. Here $H$ and $G$ are the simply connected Lie groups integrating $\h_{\delta}$ and $\g$, the map  $\bt \colon H \rightarrow G $ is the Lie group morphism integrating $\delta$, and the
  left  action $\phi$ of $G$ on $H$ (by automorphisms of $H$) is obtained integrating  the infinitesimal action (by Lie algebra derivations) $\alpha$.

\begin{description}
\item[$\bullet$ Strict Lie 2-group $(H \rtimes G)\rightrightarrows G$:]
The Lie 2-group\footnote{It is
the Lie 2-group associated to a crossed module of Lie groups   \cite{baez:2gp}.} is as follows.
Its Lie groupoid structure is the action
groupoid of the $H$ action on $G$ via  $h g := \bt(h)\cdot g$ (so  the target of
$(g, h)$ is given by  $\bt(h)\cdot g$ and its source by $g$).  The  group structure on $H \times G$ is the semidirect product structure by the action $\phi$. Explicitly, and using the notation $G_\bullet=(G_1\rightrightarrows G_0)$ for  the Lie 2-group, the group structure 
is given by
 \[m:  G_1 \times G_1 \to G_1, \quad  m( (h_1, g_1), (h_2, g_2) ) = (h_1 \cdot \phi(g_1)(h_2),  g_1g_2)  \]
over the base map $m(g_1, g_2) = g_1 g_2$.

\item[$\bullet$ $\mathcal{LA}$-group $(\h\rtimes G) \rightarrow   G$:] Its Lie algebroid structure is obtained by differentiating the Lie groupoid structure of the strict Lie 2-group $G_\bullet$, hence it is the transformation algebroid of the infinitesimal action of $\h_{\delta}$ on $G$ by $w \mapsto \overset{\rightarrow}{\delta w}$ (the right-invariant vector field  whose value at the identity is $\delta w$).

The group multiplication on $\h\times G$ is  the Lie algebroid morphism  corresponding to $m:  G_1 \times G_1 \to G_1$, i.e., $(w_1,g_1)(w_2,g_2)=(w_1+g_1w_2,g_1g_2)$. In other words, the group structure on $\h\times G$ is the semidirect product by the action of $G$ on the vector space  $\h$ obtained integrating $\alpha$.

\item[$\bullet$ Double Lie algebroid $ (\h  \rtimes \g)  \rightarrow \g $:]
Notice first that applying $Lie_{V}$ to the strict Lie 2-group above one obtains the Lie groupoid in the category of Lie algebras  $(\h_{\delta}\rtimes \g)\rightrightarrows \g$. The Lie groupoid structure is the transformation groupoid of the action of the abelian Lie algebra $\h$ on $\g$ which sends $w\in \h$ to the translation by $\delta w$.

Differentiating this Lie groupoid structure we obtain the
 Lie algebroid structure of our double Lie algebroid,
which therefore\footnote{Here we are using \cite[Thm 2.3]{Macdouble2}, which states that starting from a double Lie groupoid and applying  the functors $Lie_{H}\circ Lie_{V}$  or $Lie_{V}\circ Lie_{H}$, one obtains the same double algebroid \emph{up to} a canonical isomorphism. The canonical isomorphism in our case is the identity on $\h \times \g$. This follows from a simple computation (see the proof of \cite[Thm 2.3]{Macdouble2}) representing each element $(w,v)\in \h \times \g$ as second derivative of the map $\gamma \colon \R^2 \rightarrow H \times G, \gamma(t,u)=(exp_{\h}(tuw),exp_{\g}(tv))$.} is the transformation algebroid of the infinitesimal action of the abelian Lie algebra $\h $ on $\g$ which sends $w$ to the constant vector field $\delta w$.

The Lie algebra structure of our double Lie algebroid is obtained differentiating the Lie group structure of the $\mathcal{LA}$-group, so it  is
the semidirect product $\h  \rtimes \g$ of the action $\alpha$ of $\g$ on the vector space (abelian Lie algebra)  $\h$. Explicitly: $[(w_1,v_1),(w_2,v_2)]=([v_1, w_2]-
[v_2, w_1]\;,\; [v_1,v_2]).$

\item[$\bullet$ Q-algebra]  $(\h[1] \oplus \g, -Q_{\delta})$\textbf{:}
Applying the   functor $[1]_H$ to the double Lie algebroid we obtain the NQ-manifold $(\h[1] \oplus \g, -Q_{\delta})$. (To see this, recall that the anchor of a transformation Lie algebroid is given by the corresponding infinitesimal action, and use eq. \eqref{QA} and eq. \eqref{qd}.)
 It has the graded Lie algebra structure $\h[1] \rtimes \g$.
\end{description}

The latter $Q$-algebra structure is exactly the one corresponding to our original strict Lie-2 algebra by Lemma \ref{DGLAdoubleQ}. Hence  we obtain:
\begin{lemma} The integration of the strict Lie-2 algebra
$\h[1]\oplus \g$ is given by 
the strict Lie 2-group $(H \rtimes G)\rightrightarrows G  $ described above. 
\end{lemma}

  \begin{remark}\label{rk:courant}
As mentioned above,  a strict Lie 2-algebra can be viewed as a
 Q-algebra, or equivalently as  a double Lie algebroid, and the
integration we  perform in \S \ref{intDGLA} is its integration to a
strict Lie 2-group.

Another instance of Lie 2-algebroid that fits into the framework of
double Lie algebroids is given by a Courant algebroid $A\oplus A^*$
arising from a Lie bialgebroid $(A, A^*)$. In this case, however,
the integration of the corresponding double Lie algebroid via
diagram \eqref{rajdia} does not coincide with  the integration of
the Courant algebroid $A\oplus A^*$. The relation between the two
integrations is not yet clear to us.
\end{remark}

\subsection{Integrating strict actions: $\mathcal{LA}$-group actions on Lie algebroids and 2-group actions on Lie groupoids}\label{sec:act}

This subsection is the heart of the paper: we define the notion of global action integrating  a strict   action $\mu$ as in Def. \ref{strictaction}, 
and show that
the global action exists.
The idea is  
\begin{itemize}
\item \S \ref{whatintl2}: show that  
the transformation algebroid of the action $\mu$ fits in the framework of diagram  \eqref{rajdia} 
\item \S \ref{isa}: chase back in the diagram to obtain certain transformation groupoids, and   describe the corresponding actions.
\end{itemize}
The diagram of Mackenzie's doubles \eqref{rajdia}, in the setup at hand, is displayed in  \eqref{trafogroids} (it extends our previous diagram \eqref{rajdiaspecial}). 
 The various actions involved are displayed just before Thm. \ref{check}.

We end this subsection with remarks on the integrated actions (\S \ref{secrem}).

As earlier, we let $\h[1] \oplus \g $ be a strict Lie-2 algebra and $A\to M$ a Lie algebroid (so $A[1]$ is a NQ-1 manifold).

\subsubsection{What is an integration of a strict  Lie-2 algebra action?}\label{whatintl2}

The following proposition, which is an analogue to Lemma \ref{DGLAdoubleQ}, associates a Q-algebroid to the strict action $\mu$.

\begin{prop}\label{doublela} Let $\h[1] \oplus \g$ be a strict Lie-2 algebra, $A$ be a Lie algebroid. We consider a \emph{morphism of graded Lie algebras}
$\mu \colon \h[1] \oplus \g \rightarrow {\chi}(A[1])$. Then $\mu$ is a morphism of DGLAs (i.e. it respects differentials) if{f} the transformation algebroid of the action $\mu$
\begin{equation}
\label{dla}
\xymatrix{
 T_{\mu}:=(\h[1] \oplus \g)[1] \ar[d]\times A[1] \\
A[1],}
\end{equation}
 endowed with the homological vector field $-Q_{\delta}+Q_A$, is a $Q$-algebroid.
 Recall that  $Q_{\delta}$ was defined in Remark \ref{linqaud}, and $Q_A$ encodes the Lie algebroid structure on $A$ as in Lemma \ref{nq1la}.
\end{prop}

\begin{proof} Denote $L:=\h[1] \oplus  \g$. The vector field $-Q_{\delta}+Q_A$ is
homological  and it is linear in the fibers. So, in order to show that it is morphic,  we just have to show that, once we view it as a vector field on $L[1]\times A[1]$, it commutes with the homological vector field encoding the Lie algebroid structure of \eqref{dla}. The latter is $Q_{br}+Q_{action}$, where $Q_{br}$ was defined in Remark \ref{linqaud} and $Q_{action}$ is defined in  eq. \eqref{qaction} below
and  encodes the action $\mu$.

To show \begin{equation}\label{comm}
[-Q_{\delta}+Q_A,Q_{br}+Q_{action}]=0.
\end{equation}
we proceed as follows.

Take bases $\{v_i\}$ of $\g$ and $\{w_{\alpha}\}$ of $\h[1]$ , whose dual bases induce coordinates $\eta^i$ on $\g[1]$ and
 $P^{\alpha}$ of $\h[2]$ (of degrees $1$ and $2$ respectively).

If the differential on $L$ is given by  $\delta w_{\alpha}=D_{\alpha i} v_i$, then in coordinates
$$
Q_{\delta}= -P^{\alpha} D_{\alpha k} \pd{\eta^{k}}.$$

Further
\begin{align}\label{qaction}
Q_{action}:=\eta^i\mu(v_i)-P^{\alpha}\mu(w_{\alpha})
\end{align}
 encodes the $L$-action on $A[1]$

We have $[Q_A,Q_{br}]=0$, since the two vector fields are defined on different manifolds,
 and $[-Q_{\delta},Q_{br}]=0$ as shown in Lemma \ref{DGLAdoubleQ}. Since
$$
[Q_A,Q_{action}]=-\eta^i[Q_A,\mu(v_i)]-P^{\alpha}[Q_A,\mu(w_{\alpha})]
$$ and
$$[-Q_{\delta},Q_{action}]=P^{\alpha}D_{\alpha k}\mu(v_k)=P^{\alpha}\mu(\delta w_{\alpha})
$$
we conclude that \eqref{comm} holds if{f} $[Q_A,\mu(w_{\alpha})]=\mu(\delta w_{\alpha})$ for all $w_{\alpha}$ and $[Q_A,\mu(v_i)]=0$ for all $v_i$, which means that $\mu$ respects differentials.
\end{proof}

\begin{remark} \label{rk:non-strict}
1) Prop. \ref{doublela} together with diagram \eqref{rajdiaWHOLE} imply that
$$\big( (\h[1]\times \g)[1]\times A[1]\;,\; Q_\delta + Q_A\;,\; Q_{br}+ Q_{action} \big)$$ is   a double $Q$-manifold \cite{vorover}.
  When the action $\mu$ is not necessarily strict (in the sense of \cite{ZZL2}), 
  the above is no longer a double $Q$-manifold. However in that case one can show \cite{RajMarco} that the sum of the four above vector fields  is still a homological vector field. This makes \eqref{dla} into an ``action'' Lie 2-algebroid.   Thus  the integration of the action  should be encoded by the  Lie 2-groupoid integrating this action Lie 2-algebroid.

  In the strict case, we will see below that the integration of the action is given by a double Lie groupoid ($T_{\Phi}$ in diagram \eqref{trafogroids}). Moreover in this case, unlike the case of Courant algebroids (see Remark \ref{rk:courant}),  there is a Lie 2-groupoid, obtained applying the Artin-Mazur's codiagonal construction \cite{artin-mazur} (see Remark \ref{rk:artin-mazur}) to $T_{\Phi}$, which is also to be considered an integration of the action. These two integrating objects are not exactly equivalent: the double Lie groupoid   contains more information because the double Lie algebroid contains more information than the Lie 2-algebroid.

2)  \cite[Thm. 6.2]{rajQalg}, which is proved without the explicit use of coordinates, is a special case of Prop. \ref{doublela} (namely, the special case we consider in Ex. \ref{gA}).
\end{remark}

Now assume that
\begin{center}
\fbox{\begin{Beqnarray*}
 \mu \colon \h[1] \oplus \g &\rightarrow& {\chi}(A[1])
\end{Beqnarray*}}
\end{center}
is a strict action (Def. \ref{strictaction}), i.e. a morphism of DGLAs.
Prop. \ref{doublela} allows to view the action $\mu$ in the framework of diagram
\eqref{rajdia}. 

We make the following definitions of integration of $\mu$.
\begin{defi}\label{defin}
The \emph{integration} of the strict action $\mu$ above is either of the following actions:
\begin{itemize}
\item[I)] An $\mathcal{LA}$-group action $\Psi$ of
$(\h\rtimes G) \rightarrow   G$ on $A\rightarrow   M$
 such that applying the functor $\Lie_V$ to its transformation groupoid one obtains the Q-algebroid  \eqref{dla}.
  \item[II)] A strict Lie 2-group action $\Phi$ of  $(H \rtimes G)\rightrightarrows G$ on $\Gamma\rightrightarrows M$ such that applying the functors $\text{Lie}_H$ and  $\Lie_V$ to its transformation groupoid one obtains the Q-algebroid  \eqref{dla}.
   (Here $\Gamma\rightrightarrows M$ denotes the source simply connected Lie groupoid integrating $A$.)
\end{itemize}
\end{defi}

Def. \ref{defin} requires some explanation. For a small category $\cC$ with direct
products, a group object $G$ in $\cC$, and an object $N$ in $\cC$, the
transformation groupoid is $G\times N\rightrightarrows N$  as for usual Lie
group action,  and it is a groupoid object in $\cC$ \cite[Prop. 3.0.15]{lucathesis}.
The transformation groupoid appearing in  I) is the one corresponding to  the group action  $\Psi$.
As $\Psi$ is a group action in the category $\textbf{LA}$,
the transformation groupoid is a groupoid in $\textbf{LA}$, (an $\mathcal{LA}$-groupoid), and hence fits in diagram \eqref{rajdia}. The same reasoning, applied to the category {\bf StrLgd}, holds for II).

\emptycomment{
 \begin{prop} \label{luca}
Let $\cC$ be a small category with direct products, $G$ a group\footnote{In \cite[Prop. 3.0.15]{lucathesis} $G$ is allowed to be a \emph{groupoid} object and $\cC$ is additionally required to have fiber products.} object in $\cC$ and $N$ an object in $\cC$. Then a group action $G \times N \rightarrow N$ is a morphism in $\cC$ if{f} the corresponding transformation groupoid $G \times N \rightrightarrows N$ is a groupoid object in $\cC$.
\end{prop}
}

 The situation is summarized in   diagram \eqref{trafogroids} in \S \ref{isa},
which is
the diagram of Mackenzie's doubles \eqref{rajdia} for the above transformation groupoids. \\

Applying the vertical integration functor to a double Lie algebroid whose  vertical structures are  \emph{transformation} algebroids (such as  the one associated to \eqref{dla} above), one
expects to obtain an $\mathcal{LA}$-groupoid whose vertical structures are \emph{transformation} groupoids.

Unfortunately there are no  general enough statements about the integration of double algebroids to $\mathcal{LA}$-groupoids or $\mathcal{LA}$-groupoid to double groupoids (see Remark \ref{invert}), so we can not ``chase back'' in diagram \eqref{rajdia}, but rather in \S \ref{isa} we  have to check explicitly that the expected fact mentioned above is true for the double algebroid \eqref{dla}.

Moreover our method has the advantage that it provides explicit formulae.

\subsubsection{Integrating the strict action}\label{isa}

Now we integrate the strict Lie 2-algebra action $\mu$ on a Lie algebroid $A$ to both a $\mathcal{LA}$-group action on $A$ and a Lie 2-group action on the source-simply connected Lie groupoid of $A$.

By virtue of the identification of Remark \ref{identifyvf}, $\mu$ induces a
Lie algebra morphism
\begin{center}
\fbox{\begin{Beqnarray*}
 \tilde{\mu} \colon \h \rtimes \g \rightarrow \chi(A).
\end{Beqnarray*}}
\end{center}
Here
 $\h \rtimes  \g$ is the semidirect product of $\g$ and the abelian Lie algebra $\h$, and agrees
with the Lie algebra structure
induced by the original graded Lie algebra structure on $\h[1] \oplus \g $.
In other words, we have an infinitesimal action of the Lie algebra $\h \rtimes \g $ on $A$.
Using this identification,  $w\in \h$  maps to $\mu(w)$ seen as a constant vertical vector field on $A$ and  $v \in \g$  maps to the infinitesimal vector bundle automorphism given by $\mu(v)$.\\

Assume that the infinitesimal action $ \tilde{\mu}|_{\g} \colon \g \rightarrow \chi(A)$ is \emph{complete} in the sense that the image of $\tilde{\mu}$ consists of complete vector fields.

Then the infinitesimal action
$\tilde{\mu}|_{\g}$ can be integrated to a  Lie group action 
 $$\psi \colon G \times A \to A$$  of $G$ on $A$, where $G$ denotes the simply connected Lie group integrating $\g$.
 Notice that $\psi$ acts by Lie algebroid automorphisms of $A$, as a consequence of Lemma \ref{classical}.
  \begin{prop}\label{prop:LAgroupaction} Consider the Lie algebroids $A\to M$ and
$\h \rtimes  G \to G$ (as in \S \ref{intDGLA}).
 The Lie group action on $A$ obtained integrating the infinitesimal action $\tilde{\mu}$ is
\begin{center}
\fbox{\begin{Beqnarray*}
 \Psi  \colon (\h \rtimes  G) \times A &\rightarrow& A\\
(w,g),a_x &\mapsto& \psi(g,a_x)+\mu(w)|_{gx}
\end{Beqnarray*}}
\end{center}
where $x\in M$ and $a_x \in A_x$.

Further $\Psi$ is a Lie algebroid morphism. In other words, $\Psi$ is an $\mathcal{LA}$-group action.
\end{prop}

The proof of Prop. \ref{prop:LAgroupaction} is given in Appendix \ref{App:longproof}.

Integrating the Lie algebroid $\h \rtimes  G \rightarrow G$ we obtain the strict Lie 2-group $(H \rtimes G)\rightrightarrows G$ described in  \S \ref{intDGLA}, where $H$ is simply connected. Assume that the Lie algebroid $A\to M$ is integrable to a source simply connected Lie groupoid  $\Gamma$. 
\begin{thm}\label{thm:2groupaction}  Consider the Lie groupoids $\Gamma \rightrightarrows M$ and the Lie 2-group $(H \rtimes G)\rightrightarrows G$. 

The Lie groupoid morphism
\begin{center}
\fbox{\begin{Beqnarray*}
 \Phi  \colon (H \rtimes  G) \times \Gamma &\rightarrow& \Gamma
\end{Beqnarray*}}
\end{center}

integrating the Lie algebroid morphism $\Psi$ is a Lie 2-group action.
\end{thm}
In Prop. \ref{explicitPhi} we will give a description of the action $\Phi$.

\begin{proof} Notice that $\Phi$ is a well-defined Lie groupoid morphism as its domain is source simply connected.

We show that $\Phi$ is a group action.
Recall from \S \ref{intDGLA} that the group multiplication  $m \colon (H \rtimes  G) \times (H \rtimes  G) \rightarrow (H \rtimes  G)$ is the Lie groupoid morphism which integrates the multiplication $\tilde{m}$ on $\h \rtimes  G$. Hence integrating both sides of the equality of Lie algebroid morphisms $$\Psi \circ (\tilde{m}\times Id_A) = \Psi \circ (Id_{\h \rtimes  G} \times \Psi),$$
which holds because $\Psi$ is a group action,
one obtains $$\Phi \circ ({m}\times Id_{\Gamma})= \Phi \circ (Id_{H \rtimes  G} \times \Phi).$$
This shows that $\Phi$ is a group action both at the level  of objects and of morphisms, hence it is a Lie 2-group action.
\end{proof}

The various actions appearing in this subsection can are summarized in the following diagram:

\begin{center}
\fbox{\begin{Beqnarray*}
&\mu \colon \h[1] \oplus \g \rightarrow {\chi}(A[1]) \\
\overset{\text{                        }}  {\rightsquigarrow} & \tilde{\mu} \colon \h  \rtimes \g \rightarrow \chi(A)\\
\overset{\text{Integrate the Lie algebra action     }}{\rightsquigarrow} & \Psi  \colon (\h \rtimes  G) \times A \rightarrow A\\
\overset{\text{Integrate the Lie algebroid morphism }}{\rightsquigarrow} & \Phi  \colon (H \rtimes  G) \times \Gamma \rightarrow \Gamma
\end{Beqnarray*}}
\end{center}

 \begin{thm}\label{check}
The actions $\Phi$ and $\Psi$ defined above are really integrations of $\mu$ in the sense of Def. \ref{defin} I) and II) respectively.
\end{thm}

\begin{proof}
We consider the double Lie groupoid $T_{\Phi}$ obtained by taking the transformation groupoid of the action $\Phi$ and the  $\mathcal{LA}$-groupoid
$T_{\Psi}$ obtained by taking the transformation groupoid of the action $\Psi$.
Notice that $T_{\Phi}$ is really a double Lie groupoid by Prop. \cite[Prop. 3.0.15]{lucathesis} together with Thm. \ref{thm:2groupaction}, and
$T_{\Psi}$ is really an $\mathcal{LA}$-groupoid because of  Prop. \cite[Prop. 3.0.15]{lucathesis} together with Prop. \ref{prop:LAgroupaction}.

Further we consider the  double algebroid $T_{\tilde{\mu}}$ consisting of
the transformation algebroid
of the   action $ \tilde{\mu} \colon \h  \rtimes \g \rightarrow \chi(A)$
 and of the $\g$ action on $M$ obtained by restricting $\tilde{\mu}$ (vertically), together with $A\rightarrow M$ and the transformation algebroid of the action of the abelian Lie algebra $\h$ on $\g$ which sends $w\in \h$ to the constant vector field $\delta w$ (vertically):
\begin{equation*}
\xymatrix{
T_{\tilde{\mu}}= (\h \rtimes \g) \ar[d]\times A\ar[d] \ar[r] & \g \times M \ar[d] \\
A \ar[r]& M.}
\end{equation*}
Applying to $T_{\tilde{\mu}}$
the horizontal degree shifting functor $[1]_H$  we obtain \eqref{dla} together with the homological vector field $-Q_{\delta}+Q_A$ (see \S \ref{intDGLA}). Since by Prop. \ref{doublela} it is a $Q$-algebroid, from   Voronov's work (\cite[Thm. 1]{vorover}, see also \cite[\S 3]{vorover}) it follows in particular that
$T_{\tilde{\mu}}$ is a really a double Lie algebroid.

$T_{\Phi}, T_{\Psi}$ and
$T_{\tilde{\mu}}$ fit into the upper row of the following
table.
\begin{equation}
\label{trafogroids}
\hspace{-1.2cm}
\colorbox{mygray}{
\xymatrix{
T_{\Phi}=(H \rtimes G) \times \Gamma  \ar[r]\ar[r]\ar@<-1ex>[r] \ar[dd]\ar[dd]\ar@<-1ex>[dd]
& G  \times M \ar[dd]\ar[dd]\ar@<-1ex>[dd]
 &  T_{\Psi}=(\h\rtimes G) \times A  \ar[rr] \ar[dd]\ar[dd]\ar@<-1ex>[dd]
&& G \times M  \ar[dd]\ar[dd]\ar@<-1ex>[dd] &\\
&\ar^{\Lie_H}[r] &&&\\
 \Gamma \ar[r]\ar[r]\ar@<-1ex>[r]
&   M
& A \ar[rr]
&\ar^{\Lie}[d]
&M& \\
&
& T_{\tilde{\mu}}=(\h \rtimes \g) \times A \ar[dd]\ar[rr]
&& \g \times M \ar[dd]
& T_{\mu}=(\h[1] \rtimes \g) \ar[dd]\times A[1] \\
&&&&\ar^{[1]_H}[r] &\\
& 
& A\ar[rr]
&&M&A[1]
}
}
\end{equation}

We check that  applying the functor $\Lie_H$ to $T_{\Phi}$ we obtain $T_{\Psi}$. Clearly the horizontal Lie algebroid structures obtained this way are $A\rightarrow M$ and its product with the  Lie algebroid structure of the $\mathcal{LA}$-group $\h \rtimes G\rightarrow G$ (see \S \ref{intDGLA}). The left vertical Lie groupoid structure is obtained applying $\Lie_H$ to the maps defining the vertical Lie algebroid structures in $T_{\Phi}$. It is the transformation groupoid for the action $\Psi$, since  the Lie algebroid morphism $\Psi$ is obtained differentiating the Lie groupoid morphism $\Phi$.

Next we check that applying the functor $\Lie_V$ to $T_{\Psi}$ we obtain $T_{\tilde{\mu}}$.
Since the vertical groupoids of $T_{\Psi}$ are transformation groupoids for group actions (of $\h \rtimes G$ and $G$ respectively),
applying the functor $\Lie_V$ we obtain the transformation algebroids of the corresponding infinitesimal actions, which by Prop. \ref{prop:LAgroupaction} are $\tilde{\mu}$ and the $\g$-action on $M$ respectively.
The horizontal Lie algebroids in $T_{\Psi}$ are the Lie algebroid $A\rightarrow M$ and its product with the Lie algebroid structure on the $\mathcal{LA}$-group $\h \rtimes G\rightarrow G$. Since the application of the Lie functor to the vertical groupoids of $T_{\Psi}$ does not affect their spaces of units ($A$ and $M$ respectively), as horizontal Lie algebroid structures we obtain again $A\rightarrow M$ and its product with the Lie algebroid structure of the double Lie algebroid $\h  \rtimes \g\rightarrow \g$ (see \S \ref{intDGLA}). Altogether we obtain exactly  $T_{\tilde{\mu}}$.

Finally, we saw in the first part of this proof that applying to $T_{\tilde{\mu}}$
the horizontal degree shifting functor $[1]_H$  we obtain the $Q$-algebroid $T_{{\mu}}$ given in \eqref{dla}.

From this  we conclude that $\Psi$ and $\Phi$ are integrations of
the action $\mu$ in the sense of Def. \ref{defin}  I) and II)
respectively. 
\end{proof}

It is abstract to define the Lie 2-group action $\Phi$ in  Thm. \ref{thm:2groupaction} as the integration of the strict action $\mu$ we started with.
Now we give an explicit description of $\Phi$.

\begin{prop}\label{explicitPhi}
The Lie 2-group action $\Phi  \colon (H \rtimes  G) \times \Gamma \rightarrow \Gamma$
 can be described in terms of $\tilde{\mu}$ as follows:
\begin{itemize}
 \item[a)]
$g\in G$ acts by the Lie groupoid automorphism of $\Gamma$ which
integrates the Lie algebroid automorphism $\psi(g,\cdot)$ of $A$,
where $\psi$ denotes the Lie group action of $G$ on $A$ obtained by integrating the Lie algebra action
$\tilde{\mu}|_{\g} \colon \g \to \chi(A)$.
\item[b)] the Lie group action of $H$ on $\Gamma$ is obtained integrating the Lie algebra action $$\h_{\delta} \to \chi^{right}(\Gamma),  w \mapsto \overrightarrow{\tilde{\mu}(w)},$$
where for any section $s$ of $A\cong (ker \;\bs_*)|_M$ we denote by $\overrightarrow{s}$ its extension as a right invariant vector field on $\Gamma$.
\end{itemize}
Since a general element of $H \rtimes  G$ can be written as $(h,g)=(h,e)(e,g)$, this gives a complete description of the action $\Phi$.
\end{prop}
\begin{proof}
a) Fix $g\in G$. Consider $\Phi((e,g),\cdot) \colon \Gamma\rightarrow \Gamma$.
Its derivative restricted to $(ker \;\bs_*)|_M\cong A$ is $\Psi((0,g),\cdot) =\psi(g, \cdot) \colon A \to A$, since $\Phi$ is the Lie groupoid morphism integrating the Lie algebroid morphism $\Psi$.

b) Let $w\in \h$,
 denote by $W$ the corresponding vector field on $\Gamma$ induced by the  action $\Phi$, and let $h(t):=exp_{\h_{\delta}}(tw) \in H$. For any $m\in M$ we have  $$W_m=\frac{d}{dt}|_0 \Phi ((h(t),e),m)=\Psi((w,e),0_m)=\tilde{\mu}(w)_m,$$ so $W|_M=\tilde{\mu}(w)$.
We want to show that $W$ is a right-invariant vector field, i.e., that if $x,y \in \Gamma$ are composable elements  then $(R_y)_*W_x=W_{x\circ y}$, where $\circ$ denotes the groupoid composition on $\Gamma$. Since $\Phi$ is a Lie groupoid morphism (Thm. \ref{thm:2groupaction}) and the elements
$(h(t),e)$ and $(e,e)$ of the groupoid $H \rtimes  G$ are composable,
 we have $$\Phi(((h(t),e),x))\circ \Phi((e,e),y)=\Phi((h(t),e),x \circ y),$$ and since $\Phi((e,e),y)=y$ applying the time derivative at $t=0$ we obtain exactly $(R_y)_*W_x=W_{x\circ y}$. This shows that $W=\overrightarrow{\tilde{\mu}(w)}$ and concludes the proof.
\end{proof}
\begin{remark} Prop. \ref{explicitPhi} allows also to describe the  Lie algebra action on $\Gamma $ corresponding to the    Lie group action $\Phi$. It maps $v\in \g$ to the multiplicative vector field on $\Gamma$ which corresponds   to the vector field $\tilde{\mu}(v)$ on $A$. It maps $w\in \h$ to $\overrightarrow{\tilde{\mu}(w)}$. This shows that the action $\Phi$ coincides with the action constructed in \cite[\S 11.2]{CZ2} in the special case where $A$ is the Lie algebroid of a Poisson manifold. Notice that the action of  \cite[\S 11.2]{CZ2} was constructed ``integrating''
 $T_{\tilde{\mu}}$ (see diagram in the proof of Thm. \ref{check}) first horizontally and then vertically, while we constructed $\Phi$ ``integrating'' $T_{\tilde{\mu}}$ in the opposite order.
\end{remark}

\subsubsection{Lie 2-groupoids and principality of the integrated action}\label{secrem}

First, we make a remark on the integration of the Lie 2-algebroid \eqref{dla}.
This is the point of view that should be generalized to integrate   non-strict actions (in the sense of \cite{ZZL2}).
\begin{remark} \label{rk:artin-mazur}
We describe the strict Lie 2-groupoid integrating the Lie 2-algebroid \eqref{dla},
obtained by applying the Artin-Mazur's codiagonal construction to the  double groupoid $T_{\Phi}$ appearing in \eqref{trafogroids}. A strict Lie 2-groupoid is a groupoid object in the category of Lie groupoids whose  space of objects (object-groupoid) is just a manifold. Our strict Lie 2-groupoid has as  space of objects $M$ (the base of $A$),  as object space of the morphism-groupoid (space of 1-arrows) $G\times \Gamma$, and as morphism   space of the morphism-groupoid (space of 2-arrows) $H\times G \times \Gamma$. The space of 2-arrows is a Lie groupoid over the space of 1-arrows with source and target defined by
\[
 \bs(h, g, \gamma)=(g, \gamma), \quad \bt(h, g, \gamma) = (\delta (h)g, \Phi(\phi(g^{-1}) (h^{-1}), 1, \gamma)),
\]
and groupoid multiplication $\vertprod$ (vertical product) defined by
\[
(h, g, \gamma) \vertprod (h', g', \gamma') = (hh', g', \gamma \gamma').
\]
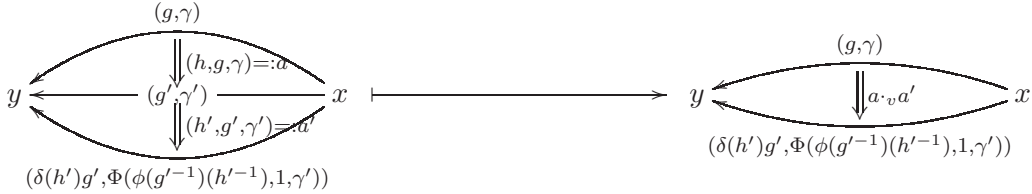
 \begin{multline*}
    \xymatrix@C+8em{
      y &
      \ar@/_2pc/[l]_{(g, \gamma)}_{}="0"
      \ar[l]|{\phantom{(} (g', \gamma') \phantom{)}}^{}="1"_{}="1b"
      \ar@/^2pc/[l]^{(\delta(h')g', \Phi(\phi(g'^{-1})(h'^{-1}), 1, \gamma'))}^{}="2"
      \ar@{=>} "0";"1"^{(h, g, \gamma)=:a}
      \ar@{=>} "1b";"2"^{(h', g', \gamma')=:a'}
      x}
    \xymatrix@C+8em{\ar@{|->}[r] & }
    \xymatrix@C+8em{
      y & \ar@/_1pc/[l]_{(g, \gamma)}_{}="4"
      \ar@/^1pc/[l]^{(\delta(h')g', \Phi(\phi(g'^{-1})(h'^{-1}), 1, \gamma'))}^{}="5"
      \ar@{=>}"4";"5"^{a\vertprod a'}
      x
    }
  \end{multline*}

Moreover, there are two maps $\bt_0, \bs_0: G\times \Gamma \to M$ defined by $\bs_0(g, \gamma)= \bs_\Gamma(\gamma)$  and $\bt_0(g, \gamma)=\psi(g, \bt_\Gamma (\gamma))$ via the source and target of $\Gamma$ (denoted as $x$ and $y$ respectively in the above picture), so that we define the horizontal multiplication $\horprod$
\[
 \xymatrix{(H\times G \times \Gamma) \times_{\bs\circ \bs_0, M, \bt \circ \bt_0}  (H\times G \times \Gamma) \ar[d] \ar@<-1ex>[d] & \ar[r]^-{\horprod} &H\times G \times \Gamma \ar[d] \ar@<-1ex>[d]  \\
(G\times \Gamma) \times_{\bs_0, M, \bt_0} (G \times \Gamma) &\ar[r]^-{\horprod} &  G \times \Gamma  },
\]
by
\[\begin{split}
&(h_1, g_1, \gamma_1) \horprod (h_2, g_2, \gamma_2) = (h_1 \phi(g_1)(h_2), g_1 g_2, \Phi(1, g_2^{-1}, \gamma_1) \gamma_2),\\ &(g_1, \gamma_1) \horprod ( g_2, \gamma_2) = ( g_1 g_2, \Phi(1, g_2^{-1}, \gamma_1) \gamma_2).
\end{split}
\]
\[
  \xymatrix@C+8em{
    z &
    \ar@/_1pc/[l]_{(g_1, \gamma_1)}_{}="0"
    \ar@/^1pc/[l]^{\bt(h_1, g_1, \gamma_1)}^{}="1"
    \ar@{=>}"0";"1"^{(h_1, g_1, \gamma_1)=:a_1}
    y &
    \ar@/_1pc/[l]_{(g_2, \gamma_2)}_{}="2"
    \ar@/^1pc/[l]^{\bt(h_2, g_2, \gamma_2 )}^{}="3"
    \ar@{=>}"2";"3"^{(h_2, g_2, \gamma_2)=:a_2}
    x}\mapsto
  \xymatrix@C+5em{
    z & \ar@/_1pc/[l]_{(g_1g_2,\Phi(1, g_2^{-1}, \gamma_1) \gamma_2) }_{}="4"
    \ar@/^1pc/[l]^{(\delta(h_1)g_1 \delta(h_2)g_2, \Phi(\phi((g_1g_2)^{-1})((h_1\phi(g_1)(h_2))^{-1}), 1, \gamma)}^{}="5"
    \ar@{=>}"4";"5"^{a_1\horprod a_2}
    x.
  }
  \]
It is routine to verify that $\horprod$ is a groupoid morphism. Thus these data define   a strict Lie 2-groupoid. It is obvious that this Lie 2-groupoid differentiates to the Lie 2-algebroid \eqref{dla}. This Lie 2-groupoid can thus be interpreted as the ``action''-groupoid of the action $\Phi$.
\end{remark}

Next we remark that, under certain assumptions, the above action $\Phi$
in Theorem \ref{thm:2groupaction} defines a \emph{principal 2-group
bundle} (over a manifold $N$) in the sense of Wockel  \cite[Def.
I.8]{Wo}.
One reason why  principal  2-group bundles are  interesting is the
following \cite[Rem. II.11]{Wo}: when the Lie 2-group $\cG$
corresponds to a crossed module of Lie groups of the form
$(H,Aut(H),\partial, \phi)$, where $H$ is a Lie group and $\partial
\colon H \to Aut(H)$ is given by conjugation, then  principal
$\cG$-2-bundles  define gerbes over $N$ \cite{CMPing}.

\begin{prop}
Let $\Phi  \colon (H \rtimes G) \times \Gamma \to \Gamma$ be a Lie
2-group action  such that both the $G$-action on $M$ and the $H
\rtimes G$-action on $\Gamma$ are free and proper with the same
quotient $N:=M/G=\Gamma/(H \rtimes G)$. Then the action $\Phi$ makes
$$\Gamma \overset{\pi}{\to}  N$$ into a principal $\cG$-2-bundle.
\end{prop}
\begin{proof}
A special case of the the above statement is \cite[Prop. 4.2]{CZ2}.
Its proof carries over literarily to the present case, after one
checks that   the projection $\Gamma \to \Gamma/(H \rtimes G)=N$ is
a Lie groupoid morphism. The  latter fact  follows easily since  the
action map $\Phi  \colon (H \rtimes G) \times \Gamma \to \Gamma$ is
a Lie groupoid morphism.
\end{proof}


\section{Examples}\label{sec:ex}
In this section we display classes of examples of Lie 2-algebra actions on  NQ-1 manifolds   (see \S \ref{sec:actions})
and of the integrated actions (see \S \ref{sec:act}).

The starting data for the first example is just a bracket-preserving map of a Lie algebra into the sections of a Lie algebroid.

\begin{ep}\label{gA}[$\g$ acting on $A$]
Let $\g$ be a Lie algebra, $A\rightarrow M$ a Lie algebroid and $\eta \colon \g \rightarrow (\Gamma(A),[\cdot,\cdot]_A)$ a Lie algebra morphism\footnote{In \cite[Def. 6.1]{rajQalg}
this is called \emph{$A$-action of $\g$ on $M$}.}.
 Then $(A[1],Q:=Q_A)$ is an NQ-1 manifold by Lemma \ref{nq1la}, and
\begin{align*}
 \mu  \colon \g[1] \oplus  \g  \rightarrow& \chi(A[1])\\
(w,0)\mapsto& \eta(w)\\
(0,\delta w)\mapsto& [Q,\eta(w)]
\end{align*}
is a morphism a DGLAs. Here $\g[1] \oplus  \g$ denotes the strict Lie-2 algebra with $\delta=Id_{\g}$ and $[\cdot,\cdot]$ given by the Lie bracket on $\g$ and its adjoint action. Further
we view $\eta(w)$ as an element of $\chi_{-1}(A[1])\cong \Gamma(A)$.

 \end{ep}

The following example is a special case of Ex. \ref{gA} for which we can write down explicitly the integrated action $\Phi$.
\begin{ep}\label{ex:TM}[$\g$ acting on $TM$]
Let $G$ be a Lie group and  $\hat{\eta}$ an action of $G$ on  a manifold $M$ (which for simplicity we take to be simply connected).
It is immediate that the product action
\begin{align}\label{simpleac}
   (G \times  G) \times (M\times M) \rightarrow& (M \times M)\\
(g_1, g_2), (m_1,m_2)  \mapsto & (g_1 m_1,g_2m_2) \nonumber
\end{align}
is a Lie 2-group action, where all the Lie groupoids appearing are pair groupoids.

We show how to recover this Lie 2-group action from infinitesimal data.
Let $\g$ be the Lie algebra of $G$, and  $\eta \colon \g \rightarrow \chi(M)$ the infinitesimal action corresponding to $\hat{\eta}$ (that is, a Lie algebra morphism $\g \to (\Gamma(TM),[\cdot,\cdot]_{TM})$).
By Ex. \ref{gA} we obtain a (strict) morphism of DGLAs
$ \mu  \colon \g[1] \oplus  \g  \rightarrow \chi(T[1]M)$.
It induces a Lie algebra morphism
\begin{align*}
 \tilde{\mu}  \colon \g_{ab} \rtimes  \g  \rightarrow& \chi(TM)\\
(w,0)\mapsto& \eta(w)\\
(0,\delta w)\mapsto& (\eta(w))^T.
\end{align*}
Here $\g_{ab}$ denotes the vector space underlying the Lie algebra $\g$ and
the first $\eta(w)$ is
viewed a vector field tangent to the fibers of $TM\to M$.
The tangent lift $(\eta(w))^T$ of the vector field $\eta(w)$ on $M$ appears since it agrees with the element $[[Q,\eta(w)],\cdot]=[\eta(w),\cdot]_{TM}$ of $CDO(TM)$
(see Lemma \ref{N1}).
By Prop. \ref{explicitPhi}
this integrates to the Lie 2-group action
\begin{align*}
 \Phi \colon (G \rtimes  G) \times (M\times M) \rightarrow& (M \times M)\\
(h,g), (m_1,m_2) \mapsto & (hgm_1,gm_2).
\end{align*}
Under the isomorphism of Lie 2-groups (over $Id_G$)
$$G\rtimes G \cong G\times G,\;\; (h,g)\mapsto (hg,g)$$
to the pair groupoid $G\times G {\rightrightarrows}{G}$, this action corresponds to \eqref{simpleac}.
\end{ep}

\begin{ep}\label{gAA}[Actions on $T^*M$]
Let $\h[1]\oplus \g$ be a strict Lie-2 algebra and $(M,\pi)$ a Poisson manifold. Since $T^*M$ is a Lie algebroid, $(\cM,Q):=(T^*[1]M, [\pi,\cdot]_S)$ is a NQ-1 manifold, where $[\cdot,\cdot]_{S}$ denotes the Schouten bracket of multivector fields.

In \cite[\S 8,9]{CZ2} Cattaneo and the first author consider
a morphism of DGLAs of the form
\begin{align}\label{fromCZ2}
\h[1] \oplus \g &\rightarrow {\chi}_{-1}(\cM)\oplus  {\chi}_0(\cM)\\
(w,v)&\mapsto \;\;\;X_{{J_0}^*w}+ X_{{J_1}^*v} \nonumber
\end{align}
where $({J_0},{J_1}): \cM \rightarrow    (\h[1] \oplus \g)^*[1]$ is a Poisson (moment) map, and discuss its reduction.

When $\h=\g$ and the differential is $Id_{\g}$,  the morphism \eqref{fromCZ2} is recovered from  our Ex. \ref{gA} with $A=T^*M$ and $\eta \colon \g \to \Gamma(T^*M), w \to -d(J^*_0v)$. Notice that in this case, as pointed out in \cite[Ex. 15.2]{CZ2},
the morphism \eqref{fromCZ2} is equivalent to an ordinary Hamiltonian

 action of $\g$ on $(M,\pi)$.
\end{ep}

Given a Lie group, we see that  the conjugation and adjoint actions fit into our framework:
\begin{ep}\label{adj}[Adjoint action]
Let $\g$ be a Lie algebra and consider the DGLA morphism
$$\mu \colon \g \rightarrow \chi(\g[1]),\;\;\;\; v \mapsto [Q,\iota_v],$$
where $\iota_v$ is the (constant) vector field corresponding to $v$ under $\g\cong \chi_{-1}(\g[1])$ and $Q$ is the homological vector field on $\g[1]$.
We now show that 
diagram \eqref{trafogroids} about transformation groupoids/algebroids applied to this example  is the following:

\begin{equation*} \xymatrix{
 \text{Transf. groupoid  of conjugation}\;\;\;\ar^{\Lie_H}[r]
&
\;\;\; \text{Transf. groupoid  of }Ad
\ar^{\Lie}[d]
\\
&
\;\;\;\;\;\;  \text{Transf. algebroid  of }ad.
\\
}
\end{equation*} 

The  Lie algebra action on $\g$ corresponding to $\mu$ reads
$$\tilde{\mu} =ad \colon \g \rightarrow End(\g)\subset \chi(\g),\;\;\;\; v \mapsto ad_v$$ by  Lemma \ref{N1}.
Integrating this Lie algebra action we  obtain essentially the adjoint action of $G$ on $\g$.
 More precisely, we obtain the Lie algebroid morphism  (see Prop. \ref{prop:LAgroupaction})
 $$ \Psi=Ad \colon (G \to G)\times (\g \to \{pt\})\to (\g \to \{pt\}).$$  Integrating this Lie algebroid morphism,   by Prop.  \ref{explicitPhi} a)  we obtain the action  of $G$ on itself by conjugation,
or more precisely, the Lie 2-group action
 $$\Phi=\text{conjugation} \colon (G \rra G)\times (G\rra \{pt\})\to (G\rra \{pt\}).$$
\end{ep}

\appendix
\section{Appendix}

In this  appendix  we prove two statements presented in the main body of the paper.

\subsection{Proof of Lemma \ref{N1} } \label{app:nq}
\begin{proof}[Proof of Lemma \ref{N1} ] By Example \ref{VBs}
\begin{align}\label{wedgeE}
\Gamma(S^{\bullet}( E^*[-1]))=\Gamma(\wedge^{\bullet}E^*)
\end{align}
is a sheaf over $M$ making $\cM:=E[1]$ into a degree $1$ N-manifold.
 On the right hand side appears the ordinary exterior product of the vector bundle $E^*$, and elements of $\wedge^{k}E^*$ are assigned degree $k$.
 
Conversely, let $M$ be a topological space and $\cO_M$ a sheaf of graded commutative algebras over $M$ as in Def. \ref{def:grm}, defining a graded manifold $\cM^{\prime}$. Then $M$ must be a smooth manifold and  the degree $1$ elements of $\cO_M$ form a locally free module over $C^{\infty}(M)$, hence sections of a vector bundle, whose dual we denote by $E$. From Def. \ref{locm} we conclude that $\cM^{\prime}\cong E[1]$.

From Def. \ref{locm} it is clear that $C(\cM)$ is generated (as a graded commutative algebra) by its elements of degree $0$ and $1$. By eq. \eqref{wedgeE} we have $C_0(\cM)=C^{\infty}(M)$ and
$C_1(\cM)=\Gamma(E^*)$.

A vector field on $\cM$, since it is a graded derivation, is determined by its action on functions of degree $0$ and $1$. Let $f_0,g_0\in C_0(\cM)$ and $f_1,g_1\in C_1(\cM)$. If $X_{-1}\in \chi_{-1}(\cM)$, it maps $C_0(\cM)$ to zero and maps $C_1(\cM)$ to $C_0(\cM)$.  Hence the graded derivation property   is simply
\begin{equation}  \label{eq:der1}
X_{-1} (f_0 g_1) = f_0 X_{-1} (g_1),
\end{equation}
so the action of $X_{-1}$  on $\Gamma(E^*)$ is $C^{\infty}(M)$-linear, i.e. given pairing with a section of $E$, and
we conclude that $\chi_{-1}(\cM)=\Gamma(E)$.

If $X_{0}\in \chi_{0}(\cM)$, then the action of $X_0$ preserves $C_0(\cM)$ as well as $C_1(\cM)$.
The graded derivation property on generators   reads
$$
  X_0(f_0 g_0) = X_0 (f_0) g_0 + f_0 X_0 (g_0),\;\;\;\;\;\;
 X_0 (f_0 g_1) = f_0 X_0 (g_1) + X_0 (f_0 ) g_1.
$$
The first equation  tells us that $X_0|_M$, defined restricting the action of $X_0$ to $C_0(\cM)$, is a vector field on $M$, and altogether we conclude that
$X_0$ is a covariant differential operator on $E^*$ with symbol $X_0|_M$. Hence
 $\chi_0(\cM)=CDO(E^*)$.

The  canonical identification $CDO(E^*)\cong CDO(E)$ obtained dualizing CDOs sends $X_{0}\in \chi_{0}(\cM)=CDO(E^*)$ to $[X_0,\cdot]\in CDO(E)$  (using the identification $\chi_{-1}(\cM)=\Gamma(E)$). Indeed eq. \eqref{dualCDO} applied to $Y^*=X_0$ reads
$$X_{-1}(X_{0} (\xi))+[X_0,X_{-1}](\xi)=X_0(X_{-1} (\xi))$$
for all $X_{-1}\in \chi_{-1}(\cM)=\Gamma(E)$ and $\xi \in C_1(\cM)=\Gamma(E^*)$.
\end{proof}

 \begin{remark}\label{identifyvf} Consider again a vector bundle $E \to M$ and $\cM=E[1]$. In Lemma \ref{N1} we saw that
$C_0(\cM)= C^{\infty}(M)$ and that $C_1(\cM)$ agrees with the fiber-wise linear functions on $E$. \cite[Lemma 8.7]{CZ2}  states that this induces
a canonical, bracket preserving identification of vector fields
\begin{align*}
{\chi}_{-1}(E[1])\cong &\{\text{vertical vector fields on $E$ which are invariant under translations in the fiber direction
}\}\\
 {\chi}_{0}(E[1])\cong& \{\text{vector fields on }E\text{ whose flow preserves the vector bundle structure}\}.
\end{align*}
\end{remark}

\subsection{Proof of Prop. \ref{prop:LAgroupaction}}\label{App:longproof}
\begin{proof}[Proof of Prop. \ref{prop:LAgroupaction}] Checking that $\Psi$ is really a group action is a straight-forward computation that uses

\begin{equation*}
\psi(g, (\mu(w))_x)=(\mu(gw))_{gx} \;\;\; \text{ for all }g\in G,w\in \h, x\in M,
\end{equation*}
 which is just the equivariance of $\mu|_{\h} \colon \h \rightarrow \chi_{-1}(A)=\Gamma(A)$ with respect to the $G$ actions on $\h$ and on the vector bundle $A$. 
The infinitesimal $\g$-equivariance holds by Lemma \ref{classical},
and as $G$ is connected  the global $G$-equivariance also holds.    A simple computation also shows that differentiating the group action $\Psi$ one obtains $\tilde{\mu}$. This proves the first part of the proposition.

Now we check that $\Psi$ preserves the anchor maps. Fix $(w,g)\in \h \rtimes  G$ and $a_x\in A_x$ (the fiber of $A$ over $x\in M$). Applying the anchor to $((w,g),a_x)$ we obtain $(\overrightarrow{\delta w}(g),\rho_A(a_x))$, and applying  the derivative of the action map $G \times M \rightarrow M$ gives
\begin{equation}\label{anch}
(\delta w)_M(gx)+ g\cdot \rho_A(a_x),
\end{equation}
 where
$(\delta w)_M$ denotes the vector field on $M$ induced by the infinitesimal action of $\delta w \in \g$ on $M$,   the dot denotes tangent lift of the action of $G$ on $M$, and $\rho_A$ is the anchor of $A$.
Now
 $$(\delta w)_M=(\mu(\delta w))|_M=[Q,\mu w]|_M=\rho_A(\mu w),$$  where the second equality (between vector fields on $A[1]$) holds because $\mu$ respects differentials, or alternatively by  Lemma \ref{classical}.
We saw that  the $G$-action $\psi$ on $A$ is by Lie algebroid automorphism, so in particular
 $\rho_A \colon A \rightarrow TM$ is $G$-equivariant and
$g\cdot \rho_A(a_x)=\rho_A(g\cdot a_x)$. Hence \eqref{anch} is equal to $\rho_A(\Psi((w,g),a_x))$, proving that $\Psi$ respects the anchor maps.\\

Checking that $\Psi$ maps the bracket $[\cdot,\cdot]_E$ on the product Lie algebroid $E:=(\h \rtimes  G) \times A$ to the bracket  $[\cdot,\cdot]_A$ on $A$ is more involved. First we remark that   $E$, as
  a vector bundle over $G\times M$, is a Whitney sum of pullback vector bundles
  $\pi_G^*(\h \rtimes  G) \oplus \pi_M^* A$, where $\pi_G$ and $\pi_M$
  are the obvious projections of $G\times M$ onto $G$ and $M$.
  We define the vector bundle automorphism
\begin{equation}\label{varphi}
\varphi\colon  \pi_M^*A\to \pi_M^*A\;,\; a_{(g,x)} \mapsto g^{-1}(a_{(g,x)})
\end{equation}
over the base diffeomorphism $(g,x)\mapsto (g, g^{-1}  x)$ of $G\times M$.
It is actually a Lie algebroid automorphism, since each $g\in G$ acts by Lie algebroid automorphisms of $A$. Notice that any $a\in \Gamma(A)$ can be pulled back to a
section of $\pi_M^*A\subset E$ (also denoted by $a$), and its image $\varphi(a)$ under $\varphi$ is given by
\begin{equation}\label{giro}
(\varphi(a))_{(g,x)}=g^{-1}\cdot a_{gx}
\end{equation}
The section $\varphi(a)\in \Gamma(E)$ is $\Psi$-projectable, and projects to $a\in \Gamma(A)$. Similarly,
for any $w\in \h$, the constant section $w\in\Gamma(\pi^*_G (\h \times G))\subset E$ projects to $\mu(w)\in \Gamma(A)$.
Sections of the form $\varphi(a)$ and $w$ span the whole of $E$, hence, by \cite[Prop. 4.3.8]{MK2}, it suffices to consider such sections. We have
$$\Psi[\varphi(a_1),\varphi(a_2)]_E=\Psi (\varphi [a_1,a_2]_A)=[a_1,a_2]_A= [\Psi \varphi a_1,\Psi \varphi a_2]_A$$
where in the first equality we used that $\varphi$ is a Lie algebroid automorphisms of $\pi_M^*A$. We have
$$\Psi[w_1,w_2]_E=\Psi([w_1,w_2]_{\delta})=\mu([w_1,w_2]_{\delta})=[\mu  w_1,\mu w_2]_A=
[\Psi w_1,\Psi w_2]_A$$ where in the third equality we used Lemma \ref{classical}.

Next we show that
\begin{equation}\label{hell}
[w,\varphi (a)]_E=\varphi ([\mu(w),a]_A)    \text{       for all }a\in \Gamma(A),w\in \h,
\end{equation}
as it will imply that
$$\Psi [\varphi (a),w]_E=\Psi \varphi ([a, \mu(w)]_A)= [a, \mu(w)]_A= [\Psi \varphi (a), \Psi w]_A$$
and thus conclude our proof.\\

To show \eqref{hell} we choose, on an open set $U$ of $M$,  a local frame of sections $a_i$ of $\Gamma(A)$. We have
\begin{equation}
(\varphi a_i)_{(g,x)}=f_i^j(g,x)a_j(x)
\end{equation}
for functions $f_i^j$ defined on $G\times U$. Here we use the Einstein summation convention.
By the Leibniz rule we can write the l.h.s. of eq. \eqref{hell} as
\begin{equation}\label{hellrhs}
([w,\varphi (a_i)]_E)_{(g,x)}=
\overset{\rightarrow}{\delta w}(f_i^k(g,x)) a_k(x),
\end{equation}
where $\overset{\rightarrow}{\delta w}$ denotes the right-invariant vector field on $G$ whose value at the identity is $\delta w$.

Notice that this lies in $\pi^*_M A\subset E$. Further, using the identification $\Gamma(A)=\chi_{-1}(A[1])$, we have
\begin{align}\label{second}
 ([\mu{w},{{a_i}}]_A)_x=&[[Q_A,\mu{w}], {a_i}]_x=[\mu(\delta {w}),  {a_i}]_x=(\cL_{\mu(\delta {w})}
{a_i})_x= \frac{d}{dt}|_0 exp(-t\delta w)\cdot (a_i)_{exp(t\delta w)\cdot x}\\
\overset{\eqref{giro}}{=}&
\frac{d}{dt}|_0 (\varphi a_i)_{(exp(t\delta w), x)}=
\frac{d}{dt}|_0 f_i^j(exp(t\delta w),x)a_j(x). \nonumber
\end{align}
We deduce that
\begin{align*}
(\varphi [\mu(w),a_i]_A)_{(g,x)}\overset{\eqref{giro}}{=}&g^{-1}\cdot ( [\mu(w),a_i]_A)_{gx}\\
\overset{\eqref{second}}{=}& \frac{d}{dt}|_0 f_i^j(exp(t\delta w),gx)\cdot g^{-1} a_j(gx)\\
\overset{\eqref{giro}}{=}&\frac{d}{dt}|_0 f_i^j(exp(t\delta w),gx)\cdot(\varphi a_j)_{(g,x)}\\
=&\frac{d}{dt}|_0 f_i^j(exp(t\delta w),gx)\cdot f_j^k(g,x) a_k(x)\\
=&\frac{d}{dt}|_0 f_i^k(exp(t\delta w)g,x)a_k(x)\\
=&  \overset{\rightarrow}{\delta w}(f_i^k(g,x)) a_k(x)\\
\overset{\eqref{hellrhs}}{=}& ([w,\varphi (a_i)]_E)_{(g,x)}.
\end{align*}
Here
 in the third last equality we used the ``multiplicativity formula''
\begin{equation}
f_i^j(gh,x)=f_i^k(g,hx)f_k^j(h,x)
\end{equation}
which can be checked writing out $(\varphi a_i)_{(gh,x)}$ by means of eq. \eqref{giro}.
Hence eq. \eqref{hell} is proved and we are done.
\end{proof}

\begin{remark}
It can be checked that the Lie algebroid automorphism $\varphi$ of $\pi^*_M A$ in \eqref{varphi} can be extended to a Lie algebroid automorphism of $E$ by asking that it maps the constant section $w$ to $w-\varphi(\mu w)\in \Gamma(E)$ for all $w\in \h$.
\end{remark}

\bibliographystyle{../bib/habbrv}
\bibliography{../bib/bibz}

\end{document}